\documentclass[a4paper,11pt]{article}
\usepackage{amsmath}
\usepackage{amssymb}
\usepackage{amsthm}
\usepackage{tikz}

\usetikzlibrary{decorations}
\usetikzlibrary{decorations.pathreplacing}
\usetikzlibrary{plotmarks}
\usetikzlibrary{intersections}

\usepackage[all]{xy}
\usepackage[latin1]{inputenc}        
\usepackage{amscd}
\usepackage{mathrsfs}
\usepackage{fullpage}
\usepackage{caption}
\usepackage{float}
\usepackage{manyfoot}

\author{Tommaso Gentile, Paolo A. Oliverio, Francesco Polizzi}
\title{On surfaces with $p_g=2, \, q=1$ and $K^2=5$}

\date{}

\newtheorem{inizio}{Lemma}[section]
\newtheorem{theorem}[inizio]{Theorem}
\newtheorem{corollary}[inizio]{Corollary}
\newtheorem{proposition}[inizio]{Proposition}
\newtheorem{remark}[inizio]{Remark}

\newtheorem{claim}[inizio]{Claim}
\newtheorem{definition}[inizio]{Definition}

\newtheorem{o-problem}[inizio]{Open Problem}

\newtheorem*{teo-L}{Theorem}

\newtheorem*{corollary-s}{Corollary}
\theoremstyle{definition}

\newcommand{\lr}{\longrightarrow}

\newcommand{\mO}{\mathcal{O}}

\begin{document}


\maketitle

\abstract We consider minimal surfaces of general type with $p_g =
2$, $q = 1$ and $K^2 = 5$. We provide a stratification of the
corresponding moduli space $\mathcal{M}$ and we give some bounds
for the number and the dimensions of its irreducible components.
\endabstract

\Footnotetext{{}}{\textit{2010 Mathematics Subject
Classification}: 14J29, 14J10}

\Footnotetext{{}} {\textit{Keywords}: Surfaces of general type,
Albanese map, genus $2$ fibration}



\section{Introduction} \label{sec:intro}
Recently there has been considerable interest in understanding the
geometry of irregular surfaces of general type. Although the
classification of such surfaces is still far from being achieved,
their study has produced in the last years a considerable amount
of results, see for instance the survey papers \cite{BCP06} and
\cite{MePa09}.

Minimal surfaces of general type satisfy the classical
inequalities:
\begin{itemize}
\item $\chi(\mO_S) := p_g - q + 1 \geq 1$, \item $K_S^2 \geq 2
p_g$ if $S$ is irregular (Debarre's inequality), \item $K_S^2 \leq
9 \chi(\mO_S)$ (Miyaoka--Yau inequality).
\end{itemize}
If $S$ is irregular and $K_S^2 = 2\chi$, then it follows $q = 1$.
In this case the Albanese map $f \colon S \lr \textrm{Alb}(S)$ is
a genus $2$ fibration whose fibres are all $2$-connected. The
corresponding classification was given by Catanese (\cite{Cat81})
for $K_S^2 = 2$, and by Horikawa (\cite{Hor82}) in the general
case.

The study of irregular surfaces with $K_S^2 = 2\chi + 1$ was
started by Catanese and Ciliberto in \cite{CaCi91} and
\cite{CaCi93}. They investigated the case $\chi = 1$, i.e., $p_g =
q = 1$ and $K_S^2 = 3$, proving that for this class of surfaces
the genus $g$ of the fibre of the Albanese map can be either $2$
or $3$. They also described all surfaces with $g = 3$ and started
the classification of surfaces with $g = 2$, which was later
completed by Catanese and Pignatelli in \cite{CaPi06}, by using a
structure theorem for genus $2$ fibrations which is proven in the
same work.

For $\chi \geq 2$ the situation is far more complicated and not
yet thoroughly studied. In this paper we consider the case $\chi =
2$, and we investigate the surfaces whose numerical invariants are
\begin{equation*}
        K_S^2 = 5, \; p_g = 2, \, q = 1.
\end{equation*}

By a result of Horikawa, given any irregular minimal surface of
general type with $2\chi \leq K^2 < \frac{8}{3}\chi$, its Albanese
map $f \colon S \lr \textrm{Alb}(S)$ is a genus $2$ fibration over
a smooth curve of genus $q$. Then in our case we have a genus $2$
fibration $f \colon S \lr B$ over an elliptic curve $B$.

We can therefore use the results of Horikawa-Xiao and those of
Catanese-Pignatelli in order to construct our surfaces and
describe their moduli space. In fact, we first study the rank $2$
vector bundle $V_1 := f_* \omega_S$, distinguishing the two cases
where $V_1$ is either decomposable or indecomposable. Then we
divide the problem in various subcases, according to the behaviour
of $V_2 := f_* \omega_S^2$, and for each subcase we study the
corresponding stratum of the moduli space $\mathcal{M}$. By
Riemann-Roch and \cite{Cle05}, at a point $[S] \in \mathcal{M}$ we
have
\begin{equation*}
        \dim_{[S]} \mathcal{M} \geq 10\chi(\mO_S) - 2 K_S^2 + p_g = 12,
\end{equation*}
hence, to understand the irreducible components of $\mathcal{M}$,
we have to consider only those strata whose dimension is greater
than or equal to $12$.

Our main results are
\begin{theorem} \label{th.01}
Let $\mathcal{M}'$ be the subspace of $\mathcal{M}$ corresponding
to surfaces such that $V_1$ is decomposable. There is  a
stratification into unirational algebraic subsets:
\begin{equation*}
        \mathcal{M}' = \mathcal{M}_{\emph{I}} \cup \mathcal{M}_{\emph{IIa}} \cup \mathcal{M}_{\emph{IIb}} \cup \mathcal{M}_{\emph{IIc}} \cup \mathcal{M}_{\emph{IIIa}}
        \cup \mathcal{M}_{\emph{IIIc}} \cup \mathcal{M}_{\emph{IVa}} \cup \mathcal{M}_{\emph{IVb}} \cup \mathcal{M}_{\emph{IVc}} \cup
        \mathcal{M}_{\emph{V}, \emph{gen}} \cup \mathcal{M}_{\emph{V}, 2},
    \end{equation*}
where $\mathcal{M}_{\emph{IIc}}$, $\mathcal{M}_{\emph{IVa}}$,
$\mathcal{M}_{\emph{IVb}}$ and $\mathcal{M}_{\emph{IVc}}$ have
dimension $\leq 11$, so they can be disregarded in the
determination of the irreducible components, while:
    \begin{description}\itemsep 2mm
      \item $\mathcal{M}_{\emph{I}}$ is nonempty, irreducible, of dimension at most $13;$
      \item $\mathcal{M}_{\emph{IIa}}$, $\mathcal{M}_{\emph{IIb}}$, $\mathcal{M}_{\emph{IIIa}}$, $\mathcal{M}_{\emph{IIIc}}$ have dimension at most $12;$
      \item $\mathcal{M}_{\emph{V}, \, \emph{gen}}$ is non-empty, of dimension $11;$
      \item $\mathcal{M}_{\emph{V}, 2}$ is a generically smooth, irreducible component of dimension $12$.
\end{description}
\end{theorem}

\begin{theorem} \label{th.02}
Let $\mathcal{M}''$ be the subspace of $\mathcal{M}$ corresponding
to surfaces such that $V_1$ is indecomposable. There is a
stratification
\begin{equation*}
        \mathcal{M}'' = \mathcal{M}_{\emph{VI}} \cup \mathcal{M}_{\emph{VIIa}} \cup \mathcal{M}_{\emph{VIIb}},
    \end{equation*}
where the strata $\mathcal{M}_{\emph{VIIa}}$ and
$\mathcal{M}_{\emph{VIIb}}$ have dimension $\leq 11$, while
$\mathcal{M}_{\emph{VI}}$ has dimension at most $12$.
\end{theorem}

Using Theorems \ref{th.01} and \ref{th.02} and some easy
additional arguments, one can prove the following

\begin{corollary}
The moduli space $\mathcal{M}$ of minimal surfaces of general type
with $p_g = 2$, $q = 1$ and $K^2 = 5$ is unirational and contains
at least $2$ irreducible components. Moreover, the dimension of
each irreducible component is either $12$ or $13$, and there is at
most one component of dimension $13$.
\end{corollary}

Of course, it would be interesting to exactly describe all
irreducible components of $\mathcal{M}$ and also to understand how
their closures intersect, but we will not try to develop this
point here.

Now let us explain how this paper is organized.

In Section 1 we present some preliminaries, and we set up notation
and terminology. In particular we recall Atiyah's classification
of vector bundles over an elliptic curve and Horikawa's and
Catanese--Pignatelli's approaches to the study of genus $2$
fibrations.

In Section 2 we investigate the structure and the possible
splitting types of the vector bundles $V_1 = f_* \omega_S$ and
$V_2 = f_* \omega_S^2$.

Finally, Section 3 deals with the study of the moduli space
$\mathcal{M}$.

\bigskip
\textbf{Acknowledgments.} The authors were partially supported by
Progetto MIUR di Rilevante Interesse Nazionale \emph{Geometria
delle Variet\`{a} Algebriche e loro Spazi di Moduli}.

They wish to thank the referee for pointing out some mistakes in a
previous version of the manuscript and for many detailed comments
that considerably improved the presentation of these results.

\section{Preliminaries}

\subsection{Vector bundles over an elliptic curve}

The classification of vector bundles of an elliptic curve was
given in \cite{At57}. Here we just recall the results needed in
order to make this paper self-contained, and we refer the reader
to Atiyah's paper for further details. Let $B$ be an elliptic
curve and let $o$ be the identity element in the group law of $B$.
If $\tau \in B$, we set $E_{\tau}(1, \,1):= \mO_B(\tau)$ and for
all $r \geq 2$ we denote by $E_{\tau}(r, \, 1)$ the unique
\emph{indecomposable} rank $r$ vector bundle on $B$ defined
recursively by the short exact sequence
\begin{equation*}
0 \lr \mO_B \lr E_{\tau}(r, \, 1) \lr E_{\tau}(r-1, \, 1) \lr 0.
\end{equation*}
Moreover, we set $F_1:=\mO_B$ and for all $r \geq 2$ we denote by
$F_r$ the unique \emph{indecomposable} rank $r$ vector bundle on
$B$ defined recursively by the short exact sequence
\begin{equation*}
0 \lr \mO_B \lr F_r \lr F_{r-1} \lr 0.
\end{equation*}

\begin{proposition}{\emph{[At57]}}\label{cohomology-fibrati}
$(i)$ For all $L \in \emph{Pic}^0(B)$ and for all $r \geq 2$ we
have
\begin{equation*}
h^0(E_{\tau}(r, \, 1)\otimes L)=1, \quad h^1(E_{\tau}(r, \,
1)\otimes L)=0.
\end{equation*}
Moreover every indecomposable rank $r$ vector bundle $V$ on $B$ such that $\deg V=1$  is isomorphic to $E_{\tau}(r, \, 1) \otimes L$ for some $L \in \emph{Pic}^0(B)$. \\
$(ii)$ For all $L \in \emph{Pic}^0(B) \setminus \{\mO_B \}$ we
have
\begin{equation*}
h^0(F_r \otimes L)=h^1(F_r \otimes L)=0,
\end{equation*}
whereas $h^0(F_r)=h^1(F_r)=1$. Moreover every indecomposable rank
$r$ vector bundle $V$ on $B$ such that $\deg V=0$ is isomorphic to
$F_r \otimes L$ for a unique $L \in \emph{Pic}^0(B)$.
\end{proposition}

By using Proposition \ref{cohomology-fibrati}, we can prove

\begin{proposition} \label{prop:W}
Let $V$ be a rank $3$ vector bundle on $B$, such that $\det  V  =
\mO_B( \tau)$ for some $\tau \in B$. Then the following holds.
\begin{itemize}
\item[$(i)$] If $h^1(V \otimes  L)=0$ for all $L \in
\emph{Pic}^0(B)$, then $V =  E_{\tau}(3, \, 1)$. \item[$(ii)$] If
$h^1(V \otimes  L)=0$ for all $ L \in \emph{Pic}^0(B) \setminus
\{\mO_B \}$ and $h^1(V)= 1$, then either $V=E_{\tau}(2, \, 1)
\oplus \mO_B$ or $V=F_2 \oplus \mO_B(\tau)$. \item[$(iii)$] If
$h^1(V \otimes  L)=0$ for all $ L \in \emph{Pic}^0(B) \setminus
\{\mO_B \}$ and $h^1(V)= 2$, then $V=\mO_B \oplus \mO_B \oplus
\mO_B(\tau)$.
\end{itemize}
\end{proposition}
\begin{proof}
$(i)$ Assume $h^1( V \otimes  L)=0$ for all $ L \in
\textrm{Pic}^0(B)$. If $ V$ is indecomposable, then $ V=
E_{\tau}(3, \, 1)$ by Atiyah's classification. Suppose now that $
V = W \oplus  M$, where $ W$ is indecomposable of rank $2$ and $
M$ is a line bundle. By our assumptions on the cohomology of $V$,
it follows $0 \leq \deg M \leq 1$. If $\deg  M =0$, then $h^1(V
\otimes  M^{-1})= 1$ yields a contradiction. If $\deg  M =1$, then
$\deg W=0$, hence $W = F_2 \otimes L$ for some $L \in
\textrm{Pic}^0 (B)$. It follows $h^1(V \otimes L^{-1})= 1$, again
a contradiction.
Finally, suppose that $V = L_1 \oplus L_2 \oplus L_3$, where the $L_i$ are line bundles. We must have $\deg L_i \geq 0$, hence we may assume $\deg L_1 =0$, $\deg L_2=0$, $\deg L_3=1$; therefore we have $h^1(V \otimes L_1^{-1}) \geq 1$, a contradiction. This concludes the proof of part $(i)$. \\
$(ii)$ Since $h^1(V)=1$, the vector bundle $V$ cannot be
indecomposable. Suppose that $ V = W \oplus  M$, where $ W$ is
indecomposable of rank $2$ and $ M$ is a line bundle; as before,
we have $0 \leq \deg M \leq 1$. If $\deg M=0$ we have $\deg W=1$,
hence $h^1(M)=h^1(V)=1$. It follows $M= \mO_B$ and $V= E_{\tau}(2,
1) \oplus \mO_B$. If $\deg M =1$ we have $\deg W =0$; since
$h^1(V)=1$, the only possibility is $V= F_2 \oplus \mO_B(\tau)$.
Finally,
suppose that $V = L_1 \oplus L_2 \oplus L_3$, where the $L_i$ are line bundles. Taking $L \in \textrm{Pic}^0(B) \setminus \{ \mO_B\}$, we have $h^1(  (L_1 \oplus L_2 \oplus L_3) \otimes L)=0$, hence $\deg L_i \geq 0$; on the other hand $\deg V =1$, hence, as before, we may assume  $\deg L_1 =0$, $\deg L_2=0$, $\deg L_3=1$; moreover $L_1 \otimes L \neq \mO_B$ and $L_2 \otimes L \neq \mO_B$ for all $L \in \textrm{Pic}^0(B) \setminus \{\mO_B \}$. Hence we obtain $V = \mO_B \oplus \mO_B \oplus \mO_B(\tau)$, so $h^1(V)=2$, a contradiction. This concludes the proof of part $(ii)$. \\
$(iii)$ Since $h^1(V)=2$, arguing as before we see that $V = L_1
\oplus L_2 \oplus L_3$, where the $L_i$ are line bundles. Moreover
$h^1(V \otimes  L)=0$ for all $ L \in \textrm{Pic}^0(B)$ implies
$\deg L_i \geq 0$. So we may assume $\deg L_1 =0$, $\deg L_2=0$,
$\deg L_3=1$, which implies $V=\mO_B \oplus \mO_B \oplus
\mO_B(\tau)$. This concludes the proof of part $(iii)$.
\end{proof}

\begin{remark} \label{rem:W}
A similar result holds if one replaces $\emph{Pic}^0(B) \setminus
\{\mO_B\}$ with $\emph{Pic}^0(B) \setminus \{M\}$, for any $M \in
\emph{Pic}^0(B)$.
\end{remark}

\begin{proposition} \label{prop:vec-ell}
\begin{itemize}
\item[$(i)$] Set $W:=E_{\tau}(2, \,1)$. Then we have
\begin{equation*}
\emph{S}^2 W = \bigoplus_{i=1}^3 L_i(\tau), \quad \emph{S}^3 W =
W(\tau) \oplus W(\tau),
\end{equation*}
where the $ L_i$ are the three non-trivial $2$-torsion line bundles on $B$. \\
\item[$(ii)$] $\textrm{S}^{r-1} F_2 = F_r,$ for all $r \geq 2$.
\end{itemize}
\end{proposition}
\begin{proof}
$(i)$ If $\tau=o$, see \cite[p. 438-439]{At57}. The general case follows since, by Proposition \ref{cohomology-fibrati}, we have $E_{\tau}(2, \, 1)=  E_o(2, \, 1) \otimes L$, where $L$ is any line bundle on $B$ such that $L^{\otimes 2} = \mO_B(\tau-o)$. \\
$(ii)$ See \cite[Theorem 9]{At57}.
\end{proof}

\subsection{Structure theorems for genus $2$ fibrations}

\subsubsection{Horikawa's method}
The following approach to genus $2$ fibrations was introduced by
Horikawa in \cite{Hor77}; see also \cite[$\S 1$]{Xi85} for further
details. Let $f \colon S \lr B$ be a relatively minimal genus $2$
fibration over a smooth curve $B$ of genus $b$, set $V_1 := f_*
\omega_{S|B}$ and let $\pi_1 \colon \mathbb{P}(V_1) \lr  B$ be the
associated $\mathbb{P}^1$-bundle. Let us consider the relative
canonical map $\phi \colon S \dashrightarrow \mathbb{P}(V_1)$,
whose indeterminacy locus is contained in the fibres of $f$ which
are not $2$-connected. After composing with a finite number of
blow-ups, we can extend $\phi$ to a generically finite, degree $2$
morphism $\tilde{\phi} \colon \widetilde{S} \lr \mathbb{P}(V_1)$;
let $\mathcal{B}$ be the branch divisor of $\tilde{\phi}$. There
exists a divisor $\mathcal{F} \in \textrm{Pic}(\mathbb{P}(V_1))$
such that $2 \mathcal{F} = \mathcal{B}$, so we can consider the
double cover $S' \lr \mathbb{P}(V_1)$ branched at $\mathcal{B}$,
and it is no difficult to see that there exists a birational
morphism $\widetilde{S} \lr S'$. The N\'{e}ron Severi group of
$\mathbb{P}(V_1)$ is generated by $C_0$ and $\Gamma$, that are the
classes of $\mO_{\mathbb{P}(V_1)}(1)$ and of a fiber,
respectively; since $\mathcal{B}\Gamma=6$, it follows that
$\mathcal{B}= 6 C_0 + \pi_1^* \alpha$, for some $\alpha \in
\textrm{Pic}(B)$. After applying a finite number of elementary
transformations to the pair $(\mathbb{P}(V_1), \, \mathcal{B})$,
we obtain that $\mathcal{B}$ has only the following types of
singularities, defined when $k \geq 1$:

\begin{itemize}
\item[$(0)$] a double point or a simple triple point;
\item[$(\mathcal{I}_k)$] a fibre $\Gamma$ plus two triple points
on it (hence these are quadruple points of $B$); each of these
triple points is $(2k-1)$-fold or $2k$-fold;
\item[$(\mathcal{II}_k)$] two triple points on a fibre, each of
these is $2k$-fold or $(2k+1)$-fold; \item[$(\mathcal{III}_k)$] a
fibre $\Gamma$ plus a $(4k-2)$ or a $(4k-1)$-fold triple point on
it which has a contact of order $6$ with $\Gamma$;
\item[$(\mathcal{IV}_k)$] a $4k$ or $(4k+1)$-fold triple point $x$
which has a contact of order $6$ with the fibre through $x$;
\item[$(\mathcal{V})$] a fibre $\Gamma$ plus a quadruple point $x$
on $\Gamma$, which a blow-up in $x$ results in a double point in
the proper transform of $\Gamma$.
\end{itemize}
We recall that a $k$-fold triple point is a triple point that
results in a simple triple point after $k-1$ blow-ups. Let us
denote by $\nu(*)$ the number of fibres of type $*$.

\begin{theorem}{\emph{\cite{Hor77}}} \label{num fibres}
The following equality holds:
\begin{equation*}
K_S^2 = 2p_a(S)-4 + 6b + \sum_k \{(2k-1)(\nu(\mathcal{I}_k) +
\nu(\mathcal{III}_k)) + 2k(\nu(\mathcal{II}_k) +
\nu(\mathcal{IV}_k)) \} + \nu(\mathcal{V}).
\end{equation*}
\end{theorem}

\subsubsection{Catanese-Pignatelli's method} \label{C-P}

Now we recall Catanese-Pignatelli approach to genus $2$
fibrations, which roughly speaking consists in considering the
relative $\emph{bicanonical}$ map insted of the canonical one. We
closely follow the treatment given in \cite{CaPi06} and
\cite{Pi09},  referring  the reader to those papers for further
details. For any relatively minimal genus $2$ fibration $f \colon
S \lr B$, we can consider the rank $3$ vector bundle $V_2 := f_*
\omega^2_{S|B}$ and the corresponding $\mathbb{P}^2$-bundle $\pi_2
\colon \mathbb{P}(V_2) \lr B$. Therefore we can associate to the
fibration $f$ the $5$-tuple $(B, \, V_1, \, \tau, \, \xi, \, w)$,
where
\begin{itemize}
\item $B$ is the base curve; \item $V_1 = f_* \omega_{S|B};$ \item
$\tau$ is an effective divisor on $B$ of degree $K_S^2 -6(b-1)- 2
\chi(\mO_S),$ corresponding to the fibres of $f$ which are not
$2$-connected; \item $\xi$ is an element of
$\textrm{Ext}^1_{\mO_B}(\mO_{\tau}, \, \textrm{S}^2 V_1)/
\textrm{Aut}_{\mO_B}(\mO_{\tau})$ giving the short exact sequence
\begin{equation} \label{eq.V}
0 \lr \textrm{S}^2 V_1 \stackrel{\sigma_2}{\lr}  V_2 \lr
\mO_{\tau} \lr 0,
\end{equation}
where $\sigma_2$ is the natural map induced by the tensor product
of canonical sections of the fibres of $f$;
 then $\sigma_2$ yields a rational map $\mathbb{P}(V_1) \dashrightarrow \mathbb{P}(V_2)$ (the relative version of the $2$-Veronese embedding $\mathbb{P}^1 \hookrightarrow \mathbb{P}^2)$ birational onto a conic bundle $\mathcal{C} \in  |\mO_{\mathbb{P}(V_2)}(2) \otimes \pi_2^*(\det V_1)^{-2}|$. More precisely, if $x_0, \, x_1$ are generators for the stalk of $V_1$, then the equation of $\mathcal{C}$ is  locally given by
\begin{equation} \label{eq:C}
\sigma_2(x_0^2) \sigma_2(x_1^2) - (\sigma_2(x_0x_1))^2 = 0.
\end{equation}
\item $w \in \mathbb{P}H^0(B, \, \widetilde{A}_6)$, where
$\widetilde{A}_6:= A_6 \otimes \big(\det V_1 \otimes
\mO_B(\tau)\big)^{-2}$ and  $A_6$ is given by the following short
exact sequence:
\begin{equation} \label{A6}
0 \lr (\det  V_1)^2 \otimes  V_2 \stackrel{i_3} \lr \textrm S^3
 V_2 \lr   A_6 \lr 0.
\end{equation}
Here the map $i_3$ is locally defined as follows: if $x_0, \, x_1$
are generators for the stalk of $V_1$ and $y_0, \, y_1, \, y_2$
are generators for the stalk of $V_2$, then
\begin{equation*}
i_3((x_0 \wedge x_1)^{\otimes 2}{\otimes \; y_i}):=
\sigma_2(x_0^2)\sigma_2(x_1^2) y_i - \sigma_2(x_0x_1)^2 y_i.
\end{equation*}
The relative bicanonical map, which is always a morphism, induces
a factorization of the fibration $f$ as
\begin{equation*}
S \stackrel{r}{\lr} X \stackrel{\psi}{\lr} \mathcal{C}
\stackrel{\pi_2|_{\mathcal{C}}}{\lr} B,
\end{equation*}
where $r$ is a contraction of $(-2)$-curves to Rational Double
Points, and $\psi$ is a finite double cover. The element $w \in
\mathbb{P}H^0(\widetilde{A}_6) = |\mO_{\mathcal{C}}(6) \otimes
(\det V_1 \otimes \mO_B(\tau) )^{-2}|$ corresponds to the
divisorial part $\Delta$ of the branch locus of $\psi$. In fact,
the branch locus of $\psi$ consists of a disjoint union $\Delta
\cup \mathcal{P}$, where $\mathcal{P} \subset
\textrm{Sing}(\mathcal{C})$ is a finite set of points in natural
bijection with $\textrm{supp}(\tau)$. Notice that $A_6$ is the
quotient of $\textrm{S}^3 V_2$ by the subbundle of the relative
cubics vanishing on $\mathcal{C}$; geometrically, this reflects
the fact that, in general, not all the divisors in
$|\mO_{\mathcal{C}}(6) \otimes (\det V_1 \otimes
\mO_B(\tau))^{-2}|$ can be written as the complete intersection of
$\mathcal{C}$ with a relative cubic $\mathcal{G} \in
|\mO_{\mathbb{P}(V_2)}(3) \otimes (\det V_1 \otimes
\mO_B(\tau))^{-2}|$. Finally, observe that if
\begin{equation} \label{A6s}
0 \lr  G_1 \lr  G_2 \lr  \widetilde{A}_6 \lr 0
\end{equation}
is the short exact sequence obtained by tensoring \eqref{A6} with
$(\det V_1 \otimes \mO_B(\tau))^{-2}$, we obtain
\begin{equation} \label{coker}
h^0(\widetilde{A}_6) \leq h^0(G_2) - h^0(G_1) + h^1(G_1).
\end{equation}
\end{itemize}
We call $(B, \, V_1, \, \tau, \, \xi, \, w)$ the \emph{associate}
$5$-\emph{ple} of the fibration $f \colon S \lr B$.

\begin{theorem}{\emph{\cite{CaPi06}}} \label{teo:genus 2}
Assume that we have a $5$-ple  $(B, \, V_1, \, \tau, \, \xi, \,
w)$ as before, such that the following $(open)$ conditions are
satisfied:
\begin{itemize}
\item[$(\mathcal{P}_1)$] the conic bundle $\mathcal{C}$ has only
Rational Double Points as singularities\emph{;}
\item[$(\mathcal{P}_2)$] the curve $\Delta$ has only simple
singularities, where ``simple" means that the germ of the double
cover of $\mathcal{C}$ branched on it has at most a Rational
Double Point.
\end{itemize}
Then there exists a unique relatively minimal genus $2$ fibration
$f \colon S \lr B$ having the above as associate $5$-ple.
Moreover, the surface $S$ has the following invariants:
\begin{equation*}
\begin{split}
\chi(\mO_S) & = \deg V_1 + (b-1), \\
K_S^2 & = 2 \deg V_1 + \deg \tau + 8(b-1).
\end{split}
\end{equation*}
\end{theorem}

\section{Surfaces of general type with $p_g=2, \, q=1$ and $K^2=5$ }

\subsection{The sheaf $V_1$}

Let $S$ be a minimal surface of general type with $p_g=2, \, q=1$
and $K_S^2=5$. Its Albanese variety $B:= \textrm{Alb}(S)$ is an
elliptic curve, and its Albanese map $f \colon S \lr B$ is a genus
$2$ fibration (\cite[Theorem 3.1]{Hor82}). Notice that since $B$
is elliptic then $\omega_{S|B}=\omega_S$. By Theorem
\ref{teo:genus 2} we have $\deg(\tau) =1$, i.e. $\tau$ is a point
of $B$. The genus $2$ fibration contains exactly one singular
fibre, which comes from a singularity of $(\mathbb{P}(V_1), \,
\mathcal{B})$ of type $\mathcal{I}_1$, $\mathcal{III}_1$ or
$\mathcal{V}$, see Theorem \ref{num fibres}. In particular, the
curve $\mathcal{B}$ contains the fibre $\Gamma_{\tau} =
\pi_1^*(\tau)$ of $\pi_1 \colon \mathbb{P}(V_1) \lr B$. Standard
calculations, see \cite[Chapter V]{BHPV04}, show that
$\mathcal{B}$ is algebraically equivalent to
$6C_0 - 2 \Gamma$, so we can write $\mathcal{B} = \mathcal{B}' + \Gamma$, where $\mathcal{B}'$ is an effective divisor algebraically equivalent to $6C_0 - 3 \Gamma$. \\
Let now $E_1$ be a rank 1 subsheaf of maximal degree of $V_1 = f_*
\omega_S$; then there is a short exact sequence
\begin{equation*}
0 \lr E \lr V_1 \lr F \lr 0
\end{equation*}
such that $F$ is locally free and $\deg F \geq 0$, see
\cite{Fu78}; moreover one clearly has $1 \leq h^0(E) \leq
h^0(V_1)=2$. Setting $e := \deg E - \deg F$, by
\cite[Th\'eor$\grave{\textrm{e}}$me 2.1 p.16]{Xi85} there are
exactly two possibilities:
\begin{itemize}
\item $\deg E=1, \quad \deg F = 1, \quad e=0$ \item $\deg E=2,
\quad \deg F = 0, \quad e=2$.
\end{itemize}

\begin{proposition} \label{prop:V1}
$(i)$ If $e=0$ then $($up to translations$)$ either $ V_1 = \mO_B(p) \oplus \mO_B(2o-p)$ for some $p \in B$ or $ V_1 = F_2(\eta)$, where $\eta \in E$ is a $2$--torsion point. \\
$(ii)$ If $e=2$ then $V_1 = \mO_B(D) \oplus L$, where $D$ is an
effective divisor of degree $2$ on $B$ and $L \in \emph{Pic}^0(B)$
is a non-trivial, torsion line bundle. This case occurs if and
only if the canonical map $\phi_{|K|}$ of $S$ factors through $f$.
\end{proposition}
\begin{proof}
$(i)$ If $e=0$, up to a translation we may assume $E=\mO_B(p)$, $F=\mO_B(2o-p)$, for some $p \in B$. If $F\neq E$, then $\textrm{Ext}^1(F, \, E)=0$ and we obtain $V_1 = \mO_B(p) \oplus \mO_B(2o-p)$. If $F=E$, then $\textrm{Ext}^1(F, \, E)=\mathbb{C}$. In that case $2o = 2p$, so any non-trivial extension class corresponds to $V_1=F_2(\eta)$, where $2\eta \in |2o|$. \\
$(ii)$ If $e=2$ then $\deg E=2$, hence $E=\mO_B(D)$ for some
effective divisor $D$ on $B$. We have $h^0(E)=2$ and $h^1(E)=0$,
so $h^0(V_1)=h^0(E)+h^0(F)$, which implies $h^0(F)=0$. Then $F$ is
a non-trivial, degree zero line bundle. Since $\textrm{Ext}^1(F,
\, E)=0$, it follows $V_1= \mO_B(D) \oplus F$, and Simpson's
results (\cite{Sim93}) imply that $F$ is a non-trivial torsion
line bundle on $B$. The last assertion follows from
\cite[Th\'eor$\grave{\textrm{e}}$me 5.1 p.71]{Xi85}.
\end{proof}

\begin{proposition} \label{no e=2}
The case $e=2$ does not occur.
\end{proposition}
\begin{proof}
If $e=2$, then $S$ would be the canonical resolution of the
singularities of a degree $2$ cover of
$\mathbb{P}(V_1)=\mathbb{P}(\mO_B(D) \oplus L)$. Since $V_1$ is
decomposable, we can take global coordinates on the fibres of
$\pi_1 \colon \mathbb{P}(V_1) \lr B$, namely
\begin{equation*}
x_0 \in H^0(\mO_{\mathbb{P}(V_1)}(1) \otimes \pi_1^* \mO_B(-D)),
\quad x_1 \in  H^0(\mO_{\mathbb{P}(V_1)}(1) \otimes \pi_1^*
L^{-1}).
\end{equation*}
Putting $M= \mO_B(D)$, we obtain $x_0^i x_1^j \in
H^0(\mO_{\mathbb{P}(V_1)}(i+j) \otimes \pi_1^* M^{-i} \otimes
\pi_1^* L^{-j})$. Since $\mathcal{B}'$ is algebraically equivalent
to $6C_0 - 3 \Gamma$, we have $\mathcal{B}' \in
|H^0(\mO_{\mathbb{P}(V_1)}(6) \otimes \pi_1^* T^{-1})|$ for a
suitable degree $3$ line bundle $T$ on $B$, so the equation of
$\mathcal{B}'$ can be written as
\begin{equation} \label{equation - B'}
\sum_{i+j=6} a_{ij} x_0^i  x_1 ^j=0,
\end{equation}
where $a_{ij} \in H^0(\mathbb{P}(V_1), \, \pi_1^*(T^{-1} \otimes
M^i \otimes L^j))$. In particular $a_{06} = a_{15}=0$, so $x_0^2$
divides the left-hand side of \eqref{equation - B'}. Hence
$\mathcal{B}'$ is non-reduced, a contradiction.
\end{proof}

Propositions \ref{prop:V1} and \ref{no e=2} imply the following

\begin{corollary} \label{canonical-no-factors}
Let $S$ be a minimal surface of general type with $p_g=2$, $q=1$,
$K_S^2=5$. Then the canonical map of $S$ does not factor through
the Albanese fibration.
\end{corollary}

\subsection{The sheaf $V_2$}

\subsubsection{The case where $V_1$ is decomposable} \label{sub:V1-dec}

If $V_1$ is decomposable then  Propositions \ref{prop:V1} and
\ref{no e=2} yield $V_1= \mO_B(p) \oplus \mO_B(2o-p)$, so we have
$\textrm{S}^2 V_1=\bigoplus_{i=1}^3 P_i$, where $P_1=\mO_B(2p)$,
$P_2= \mO_B(2o)$, $P_3=\mO_B(4o-2p)$. Fix a section $f_0 \in
H^0(\mO_B(\tau)) \setminus \{0\}$; applying the functor
$\textrm{Hom}(-, \, \textrm{S}^2 V_1)$ to the exact sequence
\begin{equation*}
0   \lr  \mO_B(o-\tau) \stackrel{(-f_0)} \lr \mO_B(o) \lr
\mO_{\tau} \lr 0
\end{equation*}
we obtain
\begin{equation} \label{eq:ext}
\textrm{Ext}^1(\mO_{\tau}, \, \textrm{S}^2 V_1)= \bigoplus_{i=1}^3
\, \frac{H^0 (P_i( \tau - o))}{H^0 (P_i(- o))} \cong \mathbb{C}^3,
\end{equation}
that is $\textrm{Ext}^1(\mO_{\tau}, \, \textrm{S}^2 V_1)$ can be
identified with the space of global sections of $\bigoplus \, H^0
(P_i( \tau - o))$, modulo the subspace of sections vanishing in
$\tau$. For any $(f_1, \, f_2, \, f_3) \in \bigoplus \, H^0
(P_i(\tau - o))$, we denote by $(\bar{f}_1, \, \bar{f}_2, \,
\bar{f}_3)$ its image in $\textrm{Ext}^1(\mO_{\tau}, \,
\textrm{S}^2 V_1)$. Arguing as in \cite[p.1032]{CaPi06}, this
implies that $V_2=f_*\omega^2_S$ is the cokernel of a short exact
sequence

\begin{equation} \label{eq-V2}
0 \lr \mO_B(o- \tau) \stackrel{i} \lr \mO_B(o) \oplus
\bigoplus_{i=1}^3 P_i \lr V_2 \lr 0,
\end{equation}

where the injective map $i$ is given by ${}^t (f_0, f_1, f_2,
f_3)$.
\begin{remark}
If we choose the map $i'$ given by ${}^t (f_0, f_1 + f_0 g_1, f_2
+ f_0 g_2, f_3 + f_0 g_3)$, with $g_i \in H^0 (P_i(-o))$, we
obtain a commutative diagram:
\begin{center}
 \begin{tikzpicture}
 \node (0left_tilde) at (0,1.2) [] {$0$};
 \node (O_tilde) at (1.7,1.2) [] {$\mO_B(o - \tau)$};
 \node (O_plus_bigplus_Pi_tilde) at (4.5,1.2) [] {$\mO(o) \oplus \bigoplus_{i=1}^3 P_i$};
 \node (V2_tilde) at (6.7,1.2) [] {$V_2'$};
 \node (0right_tilde) at (7.8,1.2) [] {$0$};

 \node (0left) at (0,0) [] {$0$};
 \node (O) at (1.7,0) [] {$\mO_B(o - \tau)$};
 \node (O_plus_bigplus_Pi) at (4.5,0) [] {$\mO(o) \oplus \bigoplus_{i=1}^3 P_i$};
 \node (V2) at (6.7,0) [] {$V_2$};
 \node (0right) at (7.8,0) [] {$0$};

 \draw (1.94,0.35) -- (1.94,0.85);
 \draw (2.06,0.35) -- (2.06,0.85);
 \draw[->] (O_plus_bigplus_Pi.north) -- node[right]{\scriptsize$M$}(O_plus_bigplus_Pi_tilde.south);
 \draw[dashed][->] (V2.north) -- (V2_tilde.south);

 \draw[->] (0left.east) -- (O.west);
 \draw[->] (O.east) -- node[above]{\scriptsize{$i$}}(O_plus_bigplus_Pi.west);
 \draw[->] (O_plus_bigplus_Pi.east) -- (V2.west);
 \draw[->] (V2.east) -- (0right.west);

 \draw[->] (0left_tilde.east) -- (O_tilde.west);
 \draw[->] (O_tilde.east) -- node[above]{\scriptsize{$i'$}}(O_plus_bigplus_Pi_tilde.west);
 \draw[->] (O_plus_bigplus_Pi_tilde.east) -- (V2_tilde.west);
 \draw[->] (V2_tilde.east) -- (0right_tilde.west);
 \end{tikzpicture}
\end{center}
where the matrix $M$ is given by $$\begin{pmatrix}
                                   1 & g_1 & g_2 & g_3 \\
                                   0 & 1 & 0 & 0 \\
                                   0 & 0 & 1 & 0 \\
                                   0 & 0 & 0 & 1 \\
                                 \end{pmatrix}$$

Hence $V_2' \cong V_2$, so the isomorphism class of $V_2$ only
depends on $(\bar{f}_1, \bar{f}_2, \bar{f}_3)$.
\end{remark}

Notice that $V_2$ is a vector bundle if and only if $f_1$, $f_2$,
$f_3$ do not vanish simultaneously in $\tau$, that is if and only
if $\xi=(\bar{f}_1, \, \bar{f}_2, \, \bar{f}_3)$ is not the
trivial extension class. Let $m$ be the cardinality of the set
$\{i \, | \, \bar{f}_i=0\}$; hence $0 \leq m \leq 2$. Now we give
the description of $V_2$ in the different cases.

\begin{proposition} \label{V2-decomp}
Assume $V_1= \mO_B(p) \oplus \mO_B(2o-p)$. Then there are
precisely the following possibilities:
\begin{itemize}
\item[\emph{(I)}] $m=0, \, \mO_B(4o-4p) \neq \mO_B, \, V_2(-2o)=
E_{\tau}(3,\,1)$ \item[\emph{(IIa)}] $m=0, \, \mO_B(4o-4p) =
\mO_B, \mO_B(2o-2p) \neq \mO_B, \, V_2(-2o)= F_2(2o-2p) \oplus
\mO_B(\tau)$ \item[\emph{(IIb)}] $m=0, \, \mO_B(4o-4p) = \mO_B,
\mO_B(2o-2p) \neq \mO_B, \, V_2(-2o)= E_{\tau}(2,\, 1) \oplus
\mO_B$ \item[\emph{(IIc)}] $m=1, \, \mO_B(4o-4p) = \mO_B,
\mO_B(2o-2p) \neq \mO_B, \, V_2(-2o)= \mO_B(2o-2p) \oplus \mO_B
\oplus \mO_B(\tau + 2p - 2o)$ \item[\emph{(IIIa)}] $m=1, \,
\mO_B(2o-2p) \neq \mO_B, \,  V_2(-2o)= E_{\tau + 2o - 2p}(2, \, 1)
\oplus \mO_B(2p -2o)$ \item[\emph{(IIIb)}] $m=1, \, \mO_B(2o-2p)
\neq \mO_B, \, V_2(-2o)= E_{\tau + 2p - 2o}(2, \, 1) \oplus
\mO_B(2o -2p)$ \item[\emph{(IIIc)}] $m=1, \, \mO_B(2o-2p) \neq
\mO_B, \, V_2(-2o) = E_{\tau}(2, 1) \oplus \mO_B$
\item[\emph{(IVa)}] $m=2, \, \mO_B(2o-2p) \neq \mO_B, \,
V_2(-2o)= \mO_B(2p-2o) \oplus \mO_B \oplus \mO_B(\tau + 2o - 2p)$
\item[\emph{(IVb)}] $m=2, \, \mO_B(2o-2p) \neq \mO_B, \, V_2(-2o)=
\mO_B(2o-2p) \oplus \mO_B \oplus \mO_B(\tau + 2p - 2o)$
\item[\emph{(IVc)}] $m=2, \, \mO_B(2o-2p) \neq \mO_B, \, V_2(-2o)=
\mO_B(2p-2o) \oplus \mO_B(2o-2p) \oplus \mO_B(\tau)$
\item[\emph{(V)}] $0 \leq m \leq 2, \, \mO_B(2o-2p) = \mO_B, \,
V_2(-2o)=\mO_B \oplus \mO_B \oplus \mO_B(\tau)$.

\end{itemize}
\end{proposition}
\begin{proof}The proof is not difficult, but one needs to consider several cases; for the reader's convenience,
we will write it in detail. Let $L \in \textrm{Pic}^{0}(B)$;
tensoring the exact sequence \eqref{eq-V2} with $L(-2o)$ we obtain
\begin{equation} \label{eq-V2-2}
0 \lr   L(-o- \tau)  \lr L(-o) \oplus  L(2p-2o) \oplus  L \oplus
L(2o-2p) \lr V_2(-2o) \otimes L \lr 0,
\end{equation}
which in turn induces a linear map in cohomology
\begin{equation*}
\alpha \colon H^1( L(-o- \tau)) \lr H^1( L(-o) \oplus  L(2p-2o)
\oplus L \oplus  L(2o-2p))
\end{equation*}
such that $H^1(V_2(-2o) \otimes L)$ is isomorphic to the cokernel
of $\alpha$. Notice that $\det V_2(-2o) = \mO_B(\tau)$. The first
component of $\alpha$ is always surjective, since it is induced by
the short exact sequence
\begin{equation*}
0 \lr L(-o- \tau) \lr L(-o) \lr \mO_{\tau} \lr 0,
\end{equation*}
therefore if  $L \notin \{\mO_B(2o-2p), \, \mO_B, \,
\mO_B(2p-2o)\}$ the map $\alpha$ is surjective and $H^1(V_2(-2o)
\otimes L)=0$. Taking the dual of $\alpha$, we obtain the map
\begin{equation*}
\alpha^* \colon H^0 \big( L^*(o) \oplus L^*(2o-2p) \oplus L^*
\oplus L^*(2p-2o) \big) \lr H^0(L^*(o+ \tau)),
\end{equation*}
which is given by $(f_0, \, f_1, \, f_2, \, f_3)$; moreover $H^1(V_2(-2o) \otimes L)^*$ is isomorphic to $\ker \alpha^*$. \\
If $\mO_B(2o - 2p) = \mO_B$, then $\alpha^*$ is injective for all
$L \in \textrm{Pic}^0(B) \setminus \{ \mO_B \}$, whereas for $L=
\mO_B$ it has a $2$-dimensional kernel; by using Proposition
\ref{prop:W} we conclude that $V_2(-2o) = \mO_B \oplus \mO_B
\oplus \mO_B(\tau)$, so we are in case $\textrm{(V)}$. Therefore
we may assume $\mO_B(2o - 2p) \neq \mO_B$. Since $\alpha^*$ is
injective unless $L \in \{\mO_B(2o - 2p), \, \mO_B, \, \mO_B(2p -
2o)\}$, we have just to consider these three cases.

If $L = \mO_B(2o - 2p)$ we obtain
\begin{equation*}
h^1(V_2(-2o) \otimes L) = \left\{
  \begin{array}{ll}
    0 & \hbox{if }  \mO_B(4o-4p) \neq \mO_B \hbox{ and } \bar{f_1} \neq 0; \\
    1 & \hbox{if }  \mO_B(4o-4p) \neq \mO_B \hbox{ and } \bar{f_1} = 0; \\
    1 & \hbox{if }  \mO_B(4o-4p) =\mO_B \hbox{ and } \bar{f_1} \neq 0 \hbox{ or } \bar{f_3} \neq 0; \\
    2 & \hbox{if }  \mO_B(4o-4p) = \mO_B \hbox{ and } \bar{f_1} = \bar{f_3} =0. \\
  \end{array}
\right.
\end{equation*}
Analogously, if $L = \mO_B(2p - 2o)$ we obtain
\begin{equation*}
h^1(V_2(-2o) \otimes L) = \left\{
  \begin{array}{ll}
    0 & \hbox{if }  \mO_B(4o-4p) \neq \mO_B \hbox{ and } \bar{f_3} \neq 0; \\
    1 & \hbox{if }  \mO_B(4o-4p) \neq \mO_B \hbox{ and } \bar{f_3} = 0; \\
    1 & \hbox{if }  \mO_B(4o-4p) =\mO_B \hbox{ and } \bar{f_1} \neq 0 \hbox{ or } \bar{f_3} \neq 0; \\
    2 & \hbox{if }  \mO_B(4o-4p) = \mO_B \hbox{ and } \bar{f_1} = \bar{f_3} =0. \\
  \end{array}
\right.
\end{equation*}
Finally, if $L = \mO_B$ we obtain
\begin{equation*}
h^1(V_2(-2o) \otimes L) = \left\{
  \begin{array}{ll}
    0 & \hbox{if }  \bar{f_2} \neq 0; \\
    1 & \hbox{if }  \bar{f_2} =0.
    \end{array}
\right.
\end{equation*}
Now we observe that if $\bar{f}_i =0$ then $P_i(-2o)$ is a direct summand of $V_2(-2o)$, and we analyze the different possibilities. \\
Assume first $\mO_B(4o - 4p ) \neq \mO_B$. In this case there
exist exactly $m$ line bundles $L$ such that $H^1(V_2(-2o)\otimes
L) \neq 0$. By a straightforward application of Proposition
\ref{prop:W} and
Remark \ref{rem:W} we obtain cases $\textrm{(I)}$, $\textrm{(IIIa)}$, $\textrm{(IIIb)}$, $\textrm{(IIIc)}$, $\textrm{(IVa)}$, $\textrm{(IVb)}$, $\textrm{(IVc)}$. \\
Now assume $\mO_B(4o - 4p ) = \mO_B$. Then the only new
possibilities are:
\begin{itemize}
\item $\bar{f}_i \neq 0$ for all $i$, that is $m=0$; then
$H^1(V_2(-2o) \otimes L)$ is trivial for all $L \in
\textrm{Pic}^0(B)$, except in the case $L = \mO_B(2o-2p) =
\mO_B(2p-2o)$ where it is $1$-dimensional. By Proposition
\ref{prop:W} and Remark \ref{rem:W} this is either
$\textrm{(IIa)}$ or $\textrm{(IIb)}$. \item $\bar{f}_1 \neq 0$,
$\bar{f}_2 = 0$, $\bar{f}_3 \neq 0$; then $H^1(V_2(-2o) \otimes
L)$ is trivial for all $L \in \textrm{Pic}^0(B)$, except in the
cases $L = \mO_B(2o-2p)$ and $L= \mO_B$ where it is
$1$-dimensional; this is $\textrm{(IIc)}$.
\end{itemize}
The proof is now complete.
\end{proof}

\subsubsection{The case where $V_1$ is indecomposable} \label{sub:V1-indec}
If $V_1$ is indecomposable, then $V_1=F_2(\eta)$, where $\eta$ is
a $2$-torsion point, so Proposition \ref{prop:vec-ell} yields
$\textrm{S}^2 V_1 = F_3(2o)$. Arguing as in Subsection
\ref{sub:V1-dec}, we obtain
\begin{equation} \label{eq:ext-indec}
\textrm{Ext}^1 (\mO_{\tau}, \, \textrm{S}^2 V_1)= \frac{H^0 (F_3(o
+ \tau))}{H^0 (F_3(o))} \cong \mathbb{C}^3,
\end{equation}
that is $\textrm{Ext}^1 (\mO_{\tau}, \, \textrm{S}^2 V_1)$ can be
identified with the space of global sections of $F_3(o + \tau)$,
modulo the subspace of sections vanishing in $\tau$. For any $v
\in H^0(F_3(o+ \tau))$, we will denote by $\bar{v}$ its image in
$\textrm{Ext}^1 (\mO_{\tau}, \, \textrm{S}^2 V_1)$. Now let us fix
a section $f_0 \in H^0(\mO_B(\tau)) \setminus \{0\}$. Then $V_2$
is the cokernel of a short exact sequence
\begin{equation} \label{eq-V2-indec}
0 \lr \mO_B(o- \tau) \stackrel{i} \lr \mO_B(o) \oplus F_3(2o) \lr
V_2 \lr 0,
\end{equation}
where the injective map $i$ is given by ${}^t (f_0, \, v)$. Notice
that $V_2$ is a vector bundle if and only if $v$ does not vanish
in $\tau$, that is if and only if $\xi:= \bar{v}$ is not the
trivial extension class. We can now give a more precise
description of $V_2$.

\begin{proposition} \label{V2-indecomp}
Assume $V_1 = F_2(\eta)$, where $\eta \in E$ is a $2$--torsion
point. Then we have the following possibilities:
\begin{itemize}
\item [$\emph{(VI)}$] $V_2(-2o) = E_{\tau}(3, \, 1)$
\item[$\emph{(VIIa)}$] $V_2(-2o)= F_2 \oplus \mO_B(\tau)$
\item[$\emph{(VIIb)}$] $V_2(-2o)= E_{\tau}(2,1) \oplus \mO_B$.
\end{itemize}
Moreover, for a general choice of $\xi \in \emph{Ext}^1(\emph{S}^2
V_1, \, \mO_{\tau})$ only $\emph{(VI)}$ occurs.
\end{proposition}
\begin{proof}

Let $L \in \textrm{Pic}^0(B)$; tensoring the exact sequence
\eqref{eq-V2-indec} with $L(-2o)$ we obtain
\begin{equation} \label{eq-V2-2-indec}
0 \lr L(-o - \tau) \lr L(-o)  \oplus \big(F_3 \otimes L  \big) \lr
V_2(-2o) \otimes L \lr 0,
\end{equation}
which in turn induces a linear map in cohomology
\begin{equation*}
\alpha \colon H^1( L(-o- \tau)) \lr H^1( L(-o)) \oplus H^1(F_3
\otimes L)
\end{equation*}
such that $H^1(V_2(-2o) \otimes L)$ is isomorphic to the cokernel
of $\alpha$. As in the proof of Proposition \ref{V2-decomp}, the
first component of $\alpha$ is always surjective. If $L \neq
\mO_B$ then $H^1(F_3 \otimes L)=0$ (see Proposition
\ref{cohomology-fibrati}); consequently, $\alpha$ is surjective
and $H^1(V_2(-2o) \otimes L)=0$. We must now investigate what
happens for $L=\mO_B$. Let $v \in \textrm{Hom}(\mO_B(-o- \tau), \,
F_3) \cong H^0(F_3(o + \tau))$, and let $Q$ be the cokernel of the
corresponding  map $v \colon \mO_B(-o- \tau) \lr F_3$.

\begin{claim} \label{claim 1}
For a general choice of $v$, we have
\begin{equation*}
Q = \mO_B(q) \oplus \mO_B(o + \tau - q)
\end{equation*}
for some $q \in B$. Moreover, $Q = \mO_B \oplus \mO_B(o + \tau)$
if and only if $\emph{im} \,v \subset W$, where $W$ is the unique
subbundle of $F_3$ isomorphic to $F_2$, see
\emph{\cite[p.433]{At57}}.
\end{claim}
\begin{proof}
Since $F_3(o+ \tau)$ is globally generated, for a general choice
of $v$ the sheaf $Q$ is locally free. If $Q$ were indecomposable
then $Q = F_2(u)$, where $u \in B$ is such that $\mO_B(2u)=
\mO_B(o + \tau)$. Since $F_r$ is self--dual, by taking duals we
obtain the exact sequence
\begin{equation*}
        0 \lr F_2(-u) \lr F_3 \lr \mO_B(o + \tau) \lr 0.
\end{equation*}
By composing it with the injective morphism $\mO_B(-u) \to
F_2(-u)$ induced by the section of $F_2$, we conclude that $\mO_B$
is a sub--vector bundle of $F_3(u)$, but this is a contradiction,
since every section of $F_3(u)$ vanishes in $u$ (see \cite[Section
5, p.108]{CaSch02}); thus $Q$ must be decomposable. Moreover, we
have $Q \cong \mO_B \oplus \mO_B(o+ \tau)$ if and only if there
exists a surjective map $F_3 \lr \mO_B$ whose kernel contains
$\textrm{im}\; v$. But such a kernel is exactly $W$, so we are
done.
\end{proof}
In order to complete the proof of Proposition \ref{V2-indecomp},
let us take a general $v \in H^0(F_3(o + \tau))$. We must then
study the exact sequence
\begin{equation*} \label{seq:to-dualize}
0 \lr \mO_B(-o - \tau) \stackrel{v}{\lr} F_3 \stackrel{j}{\lr}
\mO_B(q) \oplus \mO_B(o+ \tau - q) \lr 0,
\end{equation*}
and in particular the map $\beta$ induced in cohomology as
follows:
\begin{equation} \label{eq:beta}
H^0(\mO_B(q) \oplus \mO_B(o+ \tau - q)) \lr H^1(\mO_B(-o - \tau))
\stackrel{\beta}{\lr} H^1(F_3) \lr 0.
\end{equation}
Dualizing \eqref{eq:beta}, using Serre duality and exploiting the
isomorphism $F_3^{\ast} \cong F_3$ we obtain
\begin{equation*}
 0 \lr H^0(F_3) \stackrel{\beta^*}{\lr} H^0(\mO_B(o + \tau)) \lr H^1(\mO_B(-q) \oplus \mO_B(-o - \tau + q)),
 \end{equation*}
hence $\textrm{im} \,\beta^*$ can be identified with $\langle s_q
\rangle$, the line generated by the unique non-zero section $s_q
\in H^0(\mO_B(o + \tau))$ such that $s_q(q)=0$. Now, looking at
sequence \eqref{eq-V2-2-indec} for $L = \mO_B$, we see that
$\alpha$ is dual to
\begin{equation*}
\alpha^* \colon H^0(\mO_B(o)) \oplus H^0(F_3) \stackrel{(f_0, \,
\beta^*)}{\lr} H^0(\mO_B(o + \tau)),
\end{equation*}
so the image of $\alpha^*$ is the subspace spanned by $s_o$ and
$s_q$. Since $v$ is general we have $o \neq q$, hence $s_o$ and
$s_q$ are linearly independent sections in $H^0(\mO_B(o + \tau))$
and this implies that $\alpha^*$ is an isomorphism. Consequently,
$\alpha$ is also an isomorphism and for a general choice of $\xi =
\bar{v}$ we obtain $h^1(V_2(-2o))=0$. For some special choice of
$v \in H^0(F_3(o + \tau))$ it may happen that $\alpha^{*}$ has a
$1$-dimensional kernel, consequently, $\alpha$ has a
$1$-dimensional cokernel and $h^1(V_2(-2o))=1$. Therefore we can
apply Proposition \ref{prop:W}, concluding the proof of
Proposition \ref{V2-indecomp}.
\end{proof}

\section{The moduli space}

Let $\mathcal{M}$ be the moduli space of minimal surfaces of
general type $S$ with $p_g(S)=2$, $q(S)=1$ and $K_S^2=5$. We write
$\mathcal{M}=\mathcal{M}' \cup \mathcal{M}''$, where
$\mathcal{M}'$ corresponds to surfaces such that $ V_1$ is
decomposable and $\mathcal{M}''$ corresponds to surfaces such that $ V_1$ is indecomposable. \\

\begin{definition} \label{stratification}
We stratify $\mathcal{M}'$ and $\mathcal{M}''$ as
\begin{equation*}
\begin{split}
\mathcal{M}'& =\mathcal{M}_{\emph{I}} \cup \mathcal{M}_{\emph{IIa}} \cup \cdots \cup \mathcal{M}_{\emph{V}} \\
\mathcal{M}'' & = \mathcal{M}_{\emph{VI}} \cup
\mathcal{M}_{\emph{VIIa}} \cup \mathcal{M}_{\emph{VIIb}},
\end{split}
\end{equation*}
according to the decomposition type for $ V_2 $, as in
Propositions \emph{\ref{V2-decomp}} and \emph{\ref{V2-indecomp}}.
\end{definition}

Now we want to estimate the dimensions of these strata. By
Catanese-Pignatelli's structure theorem for genus $2$ fibrations,
we can consider a surjective map $\Phi \colon \mathcal{D} \lr
\mathcal{M}$, where $\mathcal{D}$ is the set of admissible
$5$--tuples $(B, V_1, \tau, \xi, w)$ which give surfaces with our
numerical invariants and belonging to a given stratum. Therefore
in each case the dimension of the stratum is less than or equal to
the dimension of $\mathcal{D}$.

Moreover, we will see that each strata can be parametrized via a
unirational family; therefore $\mathcal{M}$ itself is unirational.
\begin{remark} \label{rem.3.2}
In order to compute the exact dimension of each strata of the
moduli space, we must compute the dimension of the corresponding
parameter space $\mathcal{D}$, and then subtract from the result
the dimension of the general fibre of $\Phi$. Such a fibre will
correspond to the orbit of the action of certain automorphism
groups over our construction data.

Locally around the point $[S] \in \mathcal{M}$, the coarse moduli
space $\mathcal{M}$ is analytically isomorphic to the quotient of
the base $T$ of the Kuranishi family by the finite group
$\emph{Aut}(S)$. Hence
\begin{equation*}
h^1(S, \, T_S) \geq \dim_{[S]} \mathcal{M} \geq h^1(S, \, T_S) -
h^2(S, \, T_S) = 10 \chi(\mO_S) - 2 K_S^2 = 10.
\end{equation*}

When $q=1$ one obtains the better lower bound $10 \chi(\mO_S) - 2
K_S^2 + p_g = 12$, see \emph{\cite{Ran95}} and
\emph{\cite{Cle05}}. So in our case we have
\begin{equation*}
    h^1(S, \, T_S) \geq \dim_{[S]} \mathcal{M} \geq 12.
\end{equation*}
This implies that those strata whose dimension is less than $12$
can be disregarded for the determination of the irreducible
components of $\mathcal{M}$.

For further application, let us describe a method that can be used
in order to estimate $h^1(S, \, T_S)$, see \emph{\cite{Pi09}}.
There is an exact sequence
\begin{equation*}
    0 \lr \omega_S \lr \Omega_S^1 \otimes \omega_S \lr \omega_S^{\otimes 2} \lr \mO_{\emph{Crit}(f)} (\omega_S^{\otimes 2}) \lr 0,
\end{equation*}
where $f \colon S \lr B := \emph{Alb} (S)$ is the Albanese map of
$S$. Setting $\mathcal{F} := (\Omega_S^1 \otimes \omega_S) /
\omega_S$, we get
\begin{equation*}
    0 \lr \mathcal{F} \lr \omega_S^{\otimes 2} \lr \mO_{\emph{Crit}(f)} (\omega_S^{\otimes 2}) \lr 0.
\end{equation*}
Therefore
\begin{equation}\label{kernel}
            2 = h^0(S, \, \omega_S) \leq h^0(S, \, \Omega_S^1 \otimes \omega_S) \leq h^0(S, \, \omega_S) + h^0(S, \, \mathcal{F}) = 2 + h^0(S, \, \mathcal{F}),
\end{equation}
and by the Serre duality $h^2(S, \, T_S) = h^0(S, \, \Omega_S^1
\otimes \omega_S)$. Finally,
\begin{equation*}
    0 \lr H^0(S, \, \mathcal{F}) \lr H^0 (S, \, \omega_S^{\otimes 2}) \lr H^0 (S, \, \omega_S^{\otimes 2}
    \otimes \mO_{\emph{Crit}(f)}) \lr 0
\end{equation*}
implies that $H^0(S, \, \mathcal{F})$ is the vector space given by
the bicanonical curves of $S$ passing through $\emph{Crit}(f)$.
\end{remark}

Let us start by studying $\mathcal{M}'$. We have $\mO_B(p) \oplus
\mO_B(2o-p) \cong \mO_B(q) \oplus \mO_B(2o-q)$ if and only if
either $p=q$ or $p+q \in |2o|$; therefore, when $p$ varies in $B$,
the vector bundle $V_1$ varies into a $1$-dimensional family
isomorphic to $\mathbb{P}^1$.

\begin{proposition} \label{prop:I}
The stratum $\mathcal{M}_{\textrm{I}}$ is nonempty, irreducible,
of dimension at most $13$.
\end{proposition}
\begin{proof}
Set $ W :=E_{\tau}(3, \, 1)$; then $V_2= W(2o)$ and we have a
short exact sequence
\begin{equation*}
0 \lr W(2o - 2 \tau) \lr \textrm{S}^3 W (2o - 2 \tau) \lr
\widetilde{ A}_6 \lr 0,
\end{equation*}
see \eqref{A6} and \eqref{A6s}. By \cite[Section 1]{CaCi93} we
obtain
\begin{equation*}
h^0( W(2o - 2 \tau))=1, \quad  h^1( W(2o - 2 \tau))=0, \quad
h^0(\textrm{S}^3 W(2o - 2 \tau))=10,
\end{equation*}
hence $h^0(\widetilde{ A}_6)=9$. We have $1$ parameter for $B$,
$1$ parameter for $V_1$, $2$ parameters for $\xi$, $1$ parameter
for $\tau$ and $8$ parameters from $\mathbb{P}H^0(\widetilde{
A}_6)$. Therefore $\mathcal{M}_\textrm{I}$ has dimension at most
$13$, and it is irreducible since it can be parametrized via an
irreducible family.

Now let us show that it is non-empty. For the sake of simplicity,
we assume $\tau=o$ and we write $\pi \colon \mathbb{P}(W) \lr B$
and $\pi_2 \colon \mathbb{P}(V_2) \lr B$ for the projective
bundles associated to $W$ and $V_2$, respectively. There is an
isomorphism of projective bundles $\psi\colon \mathbb{P}(W) \lr
\mathbb{P}(V_2)$ such that
\begin{equation}\label{o}
\psi^* \mO_{\mathbb{P}(V_2)} (1) \cong \mO_{\mathbb{P}(W)} (1)
\otimes \pi^*\mO_B (2o).
\end{equation}
The projective bundle $\mathbb{P}(W)$ can be identified with
$\textrm{Sym}^3 B$, see for instance \cite{CaCi93}. For all $x \in
B$, set:
\begin{equation*}
\begin{split}
    D_x &= \{ x + x_2 + x_3 \mid x_2, \, x_3 \in B \}, \\
    F_x &= \{ x_1 + x_2 + x_3 \mid x_1 \oplus x_2 \oplus x_3 = x \}.
\end{split}
\end{equation*}
Then $D_o$ is the divisor class of $\mO_{\mathbb{P}(W)} (1)$, and
\eqref{o} implies that
\begin{equation}\label{*}
 \mO_{\mathbb{P}(V_2)} (1) = \mO_{\mathbb{P}(V_2)} (D_o + 2 F_o).
\end{equation}

Thus $\mathcal{C} \in |\mO_{\mathbb{P}(V_2)}(2) \otimes
\pi_2^*(\det(V_1))^{-2}| = |2 D_o + 4 F_o - 4 F_o| = |2 D_o|$.

Let now $\varphi \colon \widetilde{B} \lr B$ be an isogeny of
degree $3$, and set $G := \ker (\varphi) \cong \mathbb{Z}_3$. If
we write
\begin{equation*}
\varphi^{-1}(o) = \{\tilde{o}, \, \tilde{a}, \, \tilde{b}\},
\end{equation*}
we have $G = \langle t_{\tilde{a}}^* \rangle$, where
$t_{\tilde{a}}^*$ is the translation by $\tilde{a}$.

By \cite{At57} there exists a line bundle $L \in \textrm{Pic}
(\widetilde{B})$ of degree $1$ such that
\begin{equation*}
\varphi_* L = W
\end{equation*}
and moreover
\begin{equation}\label{oo}
\begin{split}
\varphi^* \varphi_* L = \varphi^* E_\tau (3,1) &= \mO_{\widetilde{B}}(\tilde{o}) \oplus t_{\tilde{a}}^*\mO_{\widetilde{B}}(\tilde{o}) \oplus (t_{\tilde{a}}^*)^2 \mO_{\widetilde{B}}(\tilde{o}) \\
                &= \mO_{\widetilde{B}}(\tilde{o}) \oplus \mO_{\widetilde{B}}(\tilde{a}) \oplus \mO_{\widetilde{B}}(\tilde{b}),
\end{split}
\end{equation}
see \cite[Theorem 2.2]{Is05}. Let us define $\widetilde{E} :=
\varphi^* (W \otimes \mO_{B} (2o))$; since the divisor $2
\tilde{a} + 2 \tilde{b}$ is linearly equivalent to $4 \tilde{o}$,
equation \eqref{oo} yields
\begin{equation*}
\begin{split}
\widetilde{E} &= \varphi^* W \otimes \mO_{\widetilde{B}}(2 \tilde{o} + 2\tilde{a} + 2\tilde{b})  \\
                &= \mO_{\widetilde{B}}(3\tilde{o} +2\tilde{a} + 2\tilde{b}) \oplus \mO_{\widetilde{B}}(2\tilde{o} + 3\tilde{a} + 2\tilde{b}) \oplus
                \mO_{\widetilde{B}}(2\tilde{o} + 2\tilde{a} +
                3\tilde{b}) \\
                &=\mO_{\widetilde{B}}(7 \tilde{o}) \oplus \mO_{\widetilde{B}}(6\tilde{o} + \tilde{a}) \oplus
                \mO_{\widetilde{B}}(6 \tilde{o} + \tilde{b}).
\end{split}
\end{equation*}

From the commutative diagram
\begin{center}
 \begin{tikzpicture}
 \node (P_tilde) at (0,1.2) [] {$\mathbb{P}(\widetilde{E})$};
 \node (B_tilde) at (0,0) [] {$\widetilde{B}$};
 \node (P) at (1.7,1.2) [] {$\mathbb{P}(V_2)$};
 \node (B) at (1.7,0) [] {$B$};

 \draw[->] (P_tilde.south) -- node[left]{\scriptsize{$\tilde{\pi}_2$}}(B_tilde.north);
 \draw[->] (P_tilde.east) -- node[above]{\scriptsize{$\Phi$}}(P.west);
 \draw[->] (P.south) -- node[right]{\scriptsize{$\pi_2$}}(B.north);
 \draw[->] (B_tilde.east) -- node[above]{\scriptsize{$\varphi$}}(B.west);
 \end{tikzpicture}
\end{center}
it follows
\begin{equation*}
\begin{split}
\Phi_* \Phi^* \mO_{\mathbb{P} (V_2)} (D_o) &= \mO_{\mathbb{P} (V_2)} (D_o) \otimes \Phi_* \mO_{\mathbb{P} (\widetilde{E})} \\
                  &= \mO_{\mathbb{P} (V_2)} (D_o) \otimes (\mO_{\mathbb{P}(V_2)} \oplus \mathcal {L} \oplus \mathcal {L}^2) \\
                  &= \mO_{\mathbb{P}(V_2)}(D_o) \oplus (\mO_{\mathbb{P} (V_2)} (D_o) \otimes \mathcal {L}) \oplus (\mO_{\mathbb{P} (V_2)} (D_o) \otimes \mathcal {L}^2),
\end{split}
\end{equation*}
where $\mathcal{L}$ is the $3$--torsion line bundle inducing the
\'{e}tale $\mathbb{Z}_3$-cover $\Phi \colon \mathbb{P}
(\widetilde{E}) \lr \mathbb{P}(V_2)$. By \eqref{*} we see that
\begin{equation*}
\begin{split}
\Phi^* \mO_{\mathbb{P} (V_2)} (D_o) &= \Phi^* (\mO_{\mathbb{P}(V_2)} (1) \otimes \pi_2^* \mO_B (-2o)) \\
                                    &= \mO_{\mathbb{P} (\widetilde{E})} (1) \otimes \tilde{\pi}_2^* \mO_{\widetilde{B}} (-6\tilde{o}).
\end{split}
\end{equation*}

Let $y_0$, $y_1$ and $y_2$ be global coordinates on the fibers of
$\tilde{\pi}_2$, namely
\begin{equation*}
\begin{split}
&y_0 \in H^0 (\mO_{\mathbb{P}(\widetilde{E})} (1) \otimes \tilde{\pi}_2^* \mO_{\widetilde{B}} (-7\tilde{o}) \\
&y_1 \in H^0 (\mO_{\mathbb{P}(\widetilde{E})} (1) \otimes \tilde{\pi}_2^* \mO_{\widetilde{B}} (-6\tilde{o} - \tilde{a}) \\
&y_2 \in H^0 (\mO_{\mathbb{P}(\widetilde{E})} (1) \otimes
\tilde{\pi}_2^* \mO_{\widetilde{B}} (-6\tilde{o} - \tilde{b})).
\end{split}
\end{equation*}

We have $h^0 (\Phi^* \mO_{\mathbb{P} (V_2)} (D_o)) = 3$ and a
general section of $\Phi^* \mO_{\mathbb{P} (V_2)} (D_o)$ can be
written as
\begin{equation*}
\sigma = \lambda_0 y_0 + \lambda_1 y_1 + \lambda_2 y_2,
\end{equation*}
where $\lambda_0 \in H^0 (\tilde{\pi}_2^*
\mO_{\widetilde{B}}(\tilde{o}))$, $\lambda_1 \in H^0
(\tilde{\pi}_2^* \mO_{\widetilde{B}}(\tilde{a}))$ and $\lambda_2
\in H^0 (\tilde{\pi}_2^* \mO_{\widetilde{B}}(\tilde{b}))$.

Then a straightforward computation shows that we can choose the
$y_i$ so that the action of $t_{\tilde{a}}^* \in G$ on the $y_i$
is given by
\begin{equation} \label{translation}
t_{\tilde{a}}^*\colon \left\{
    \begin{array}{l}
         y_0 \mapsto y_1 \\
         y_1 \mapsto y_2 \\
         y_2 \mapsto y_0.
    \end{array}
       \right.
\end{equation}

Therefore $t_{\tilde{a}}^* \sigma = (t_{\tilde{a}}^* \lambda_0)
y_1 + (t_{\tilde{a}}^* \lambda_1) y_2 + (t_{\tilde{a}}^*
\lambda_2) y_0$, so $\sigma$ is $G$--invariant if and only if
$t_{\tilde{a}}^* \lambda_0 = \lambda_1$, $t_{\tilde{a}}^*
\lambda_1 = \lambda_2$ and $t_{\tilde{a}}^* \lambda_2 =
\lambda_0$. Since $(t_{\tilde{a}}^*)^2 = t_{\tilde{b}}^*$, this is
equivalent to require $\lambda_1 = t_{\tilde{a}}^* \lambda_0$ and
$\lambda_2 = t_{\tilde{b}}^* \lambda_0$. So a general invariant
section of $\Phi^* \mO_{\mathbb{P} (V_2)} (D_o)$ is given by
\begin{equation*}
\lambda y_0 + (t_{\tilde{a}}^* \lambda) y_1 + (t_{\tilde{b}}^*
\lambda) y_2,
\end{equation*}
where $\lambda \in H^0(\mO_{\widetilde{B}}(\tilde{o}))$.

Now a general section of $\Phi^* \mO_{\mathbb{P} (V_2)} (2D_o)$ is
of the form:
\begin{equation*}
\begin{split}
\sigma &= \sum_{i+j+k=2} \lambda_{ijk} y_0^i y_1^j y_2^k \\
       &= \lambda_{200} y_0^2 + \lambda_{020} y_1^2 + \lambda_{002} y_2^2 + \lambda_{110} y_0 y_1 + \lambda_{101} y_0 y_2 + \lambda_{011} y_1 y_2,
\end{split}
\end{equation*}
where the $\lambda_{ijk}$ are sections of pullbacks of suitable
line bundles on $\widetilde{B}$.

By (\ref{translation}), $t_{\tilde{a}}^*$ acts on $\sigma$ as
\begin{equation*}
t_{\tilde{a}}^* \sigma = (t_{\tilde{a}}^* \lambda_{200}) y_1^2 +
(t_{\tilde{a}}^* \lambda_{020}) y_2^2 + (t_{\tilde{a}}^*
\lambda_{002}) y_0^2 + (t_{\tilde{a}}^* \lambda_{110}) y_1 y_2 +
(t_{\tilde{a}}^* \lambda_101) y_0 y_1 + (t_{\tilde{a}}^*
\lambda_{011}) y_0 y_2,
\end{equation*}
so $\sigma$ is $G$--invariant if and only if
\begin{equation*}
\begin{split}
&\lambda_{020} = t_{\tilde{a}}^* \lambda_{200}, \quad \lambda_{002} = t_{\tilde{a}}^* \lambda_{020} = t_{\tilde{b}}^* \lambda_{200}, \\
&\lambda_{011} = t_{\tilde{a}}^* \lambda_{110}, \quad
\lambda_{101} = t_{\tilde{a}}^* \lambda_{011} = t_{\tilde{b}}^*
\lambda_{110}.
\end{split}
\end{equation*}
Hence a general invariant section of $\Phi^* \mO_{\mathbb{P}
(V_2)} (2D_o)$ can be written as
\begin{equation}\label{inv.quad.}
\lambda y_0^2 + (t_{\tilde{a}}^* \lambda) y_1^2 + (t_{\tilde{b}}^*
\lambda) y_2^2 + \mu y_0 y_1 + (t_{\tilde{b}}^* \mu) y_0 y_2 +
(t_{\tilde{a}}^* \mu) y_1 y_2,
\end{equation}
with $\lambda \in H^0 (\mO_{\widetilde{B}} (2 \tilde{o}))$, $\mu
\in H^0(\mO_{\widetilde{B}} (\tilde{o} + \tilde{a}))$.

Denoting by $\tilde{p} \in \widetilde{B}$ any of the points in
$\varphi^{-1}(p)$, the short exact sequence \eqref{eq.V} lifts to
\begin{equation} \label{eq.pullback.sigma2}
0 \lr \mO_{\widetilde{B}}(6 \tilde{p}) \oplus
\mO_{\widetilde{B}}(6\tilde{o}) \oplus
                \mO_{\widetilde{B}}(12 \tilde{o} -6 \tilde{p})
                \stackrel{\tilde{\sigma}_2}\lr \widetilde{E} \lr \mO_{\tilde{o}+\tilde{a} +
                \tilde{b}} \lr 0.
\end{equation}
Taking global coordinates $\tilde{x}_0$, $\tilde{x}_1$ on the
fibres of $\varphi^*V_1=\mO_{\widetilde{B}}(3 \tilde{p}) \oplus
\mO_B(6 \tilde{o} - 3 \tilde{p})$, the map $\tilde{\sigma}_2$ is
given by
\begin{equation*}
\left\{
  \begin{array}{ll}
    \tilde{\sigma}_2(\tilde{x}_0^2)= a_{00} y_0 + a_{01}y_1 + a_{02} y_2 \\
    \tilde{\sigma}_2(\tilde{x}_0 \tilde{x}_1)= a_{10} y_0 + a_{11}y_1 + a_{12} y_2  \\
    \tilde{\sigma}_2(\tilde{x}_1^2)= a_{20} y_0 + a_{21}y_1 + a_{22}
y_2,
  \end{array}
\right.
\end{equation*}
where
\begin{equation*}
\begin{split}
a_{00} & \in H^0(\tilde{\pi}_2^* \mO_{\widetilde{B}}(7 \tilde{o} -
6 \tilde{p})), \quad a_{01} \in H^0(\tilde{\pi}_2^*
\mO_{\widetilde{B}}(6 \tilde{o} - 6 \tilde{p} + \tilde{a})), \quad
a_{02}  \in H^0(\tilde{\pi}_2^* \mO_{\widetilde{B}}(6 \tilde{o} -
6
\tilde{p}+ \tilde{b})), \\
a_{10} & \in H^0(\tilde{\pi}_2^* \mO_{\widetilde{B}}(\tilde{o})),
\quad \quad \quad \quad a_{11}  \in
H^0(\tilde{\pi}_2^*\mO_{\widetilde{B}}(\tilde{a})), \quad  \quad
\quad \quad \quad a_{12}
\in H^0(\tilde{\pi}_2^* \mO_{\widetilde{B}}(\tilde{b})), \\
a_{20} & \in H^0(\tilde{\pi}_2^* \mO_{\widetilde{B}}(6 \tilde{p} -
5 \tilde{o})), \quad a_{21} \in H^0(\tilde{\pi}_2^*
\mO_{\widetilde{B}}(6 \tilde{p} - 6 \tilde{o} + \tilde{a})), \quad
a_{22}  \in H^0(\tilde{\pi}_2^* \mO_{\widetilde{B}}(6 \tilde{p} -
6 \tilde{o}+ \tilde{b})).
\end{split}
\end{equation*}
Let us consider now the conic bundle $\widetilde{\mathcal{C}}
\subset \mathbb{P}(\widetilde{E})$ given by
\begin{equation*}
(a_{00} y_0 + a_{01}y_1 + a_{02} y_2)(a_{20} y_0 + a_{21}y_1 +
a_{22})-(a_{10} y_0 + a_{11}y_1 + a_{12} y_2)^2=0.
\end{equation*}
If we choose
\begin{equation*}
a_{01}=t_{\tilde{a}}^*a_{00}, \; \; a_{02}=t_{\tilde{b}}^*a_{00},
\quad a_{11}=t_{\tilde{a}}^*a_{10}, \; \;
a_{12}=t_{\tilde{b}}^*a_{10}, \quad a_{21}=t_{\tilde{a}}^*a_{20},
\; \; a_{22}=t_{\tilde{b}}^*a_{20}
\end{equation*}
the equation of $\widetilde{\mathcal{C}}$ is $G$-invariant, hence
of the form \eqref{inv.quad.}; in fact, we have
\begin{equation*}
\lambda = a_{00}a_{20}-a_{10}^2, \quad  \mu=
a_{00}(t_{\tilde{a}}^*a_{20})+(t_{\tilde{a}}^*a_{00})a_{20}-2a_{10}(t_{\tilde{a}}^*a_{10}).
\end{equation*}
We claim that, for a general choice of $a_{00}$, $a_{10}$,
$a_{20}$, the only singularities of $\widetilde{\mathcal{C}}$ are
three rational double points of type $A_1$, lying over the three
points $\tilde{o}$, $\tilde{a}$, $\tilde{b}$. Since
$\tilde{\sigma}_2$ is of maximal rank outside these points, and
since they form an orbit for the $G$-action, it is sufficient to
check that the fibre over $\tilde{o}$ has a node (which will be
automatically a point of type $A_1$ for
$\widetilde{\mathcal{C}})$. In a neighborhood of this fibre, set
\begin{equation*}
\begin{split}
u_0 & :=a_{00}(\tilde{o}) y_0 + a_{01}(\tilde{o})y_1 +
a_{02}(\tilde{o}) y_2, \\ u_1 & :=a_{10}(\tilde{o}) y_0 +
a_{11}(\tilde{o})y_1 + a_{12}(\tilde{o}) y_2, \\
u_2 & :=a_{20}(\tilde{o}) y_0 + a_{21}(\tilde{o})y_1 +
a_{22}(\tilde{o}) y_2.
\end{split}
\end{equation*}
Since $\tilde{\sigma}_2$ drops rank in $\tilde{o}$, we can find
$c_0, \,c_2 \in \mathbb{C}$ such that $u_1=c_0 u_0 + c_2 u_2$;
then a local equation of the fibre of $\widetilde{\mathcal{C}}$
over $\tilde{o}$ is given by
\begin{equation} \label{eq.u}
u_0u_2-(c_0u_0+c_2u_2)^2=0.
\end{equation}
Since for a general choice of $a_{00}$, $a_{10}$, $a_{20}$ (i.e.
for a general choice of $c_0$, $c_2$) the quadratic form
\eqref{eq.u} splits into two \emph{distinct} linear forms, our
claim is proven.

Therefore the image of $\mathcal{\widetilde{C}}$ in
$\mathbb{P}(V_2)$ is a conic bundle $\mathcal{C}$ with a unique
singular point of type $A_1$, lying over the point $o \in B$.
Moreover, by construction, $\mathcal{C}$ is the conic bundle
associated with the map $\sigma_2 \colon S^2 V_1 \to V_2$, so
condition $(\mathcal{P}_1)$ of Theorem \ref{teo:genus 2} is
satisfied.

The relative cubic $\mathcal{G}$ belongs to the linear system
$|\mO_{\mathbb{P}(V_2)}(3) \otimes \pi_2^* \mO_B(-4o - 2\tau)| =
|3D_o + 6F_o - 6F_o| = |3D_o|$. By \cite{CaCi93}, the linear
system $|3 D_o|$ is base point free, hence its restriction to
$\mathcal{C}$ is base point free too. This implies that a general
complete intersection of the form $\mathcal{G} \cap \mathcal{C}$
is smooth and does not contain the unique singular point of
$\mathcal{C}$. Thus condition $(\mathcal{P}_2)$ is also satisfied,
and consequently $\mathcal{M}_\textrm{I}$ is not empty.
 \end{proof}

\begin{proposition}
The stratum $\mathcal{M}_{\emph{IIa}}$ has dimension at most $12$.
\end{proposition}
\begin{proof}
In case (IIa) we have $\mO_B(4o-4p)=\mO_B$, so there are no
parameters for $V_1$. The vector bundle $\widetilde{A}_6$ fits
into the short exact sequence
\begin{equation*}
0 \lr G_1 \lr G_2 \lr \widetilde{A}_6 \lr 0,
\end{equation*}
where
\begin{equation*}
\begin{split}
G_1 & = F_2(2p - 2\tau) \oplus  \mO_B(2o - \tau), \\
G_2 & = \textrm{S}^3 F_2(2p-2\tau) \oplus \textrm{S}^2F_2(2o-\tau)
\oplus F_2(2p) \oplus \mO_B(2o+\tau).
\end{split}
\end{equation*}
By Proposition \ref{prop:vec-ell} we have $\textrm{S}^2 F_2 = F_3,
\;\textrm{S}^3 F_2 = F_4$.
Now there are two possibilities. \\

$\bullet$ $\mO_B(2p - 2\tau) \neq \mO_B$. In this case
\begin{equation*}
h^0(G_1)=1, \quad h^1(G_1)=0, \quad h^0(G_2)=10,
\end{equation*}
hence $h^0(\widetilde{A}_6) = h^0(G_2) - h^0(G_1)=9$. We have $1$
parameter for $B$, $2$ parameters for $\xi$, $1$ parameter for
$\tau$ and $8$ parameters from $\mathbb{P}H^0(\widetilde{A}_6)$. \\

$\bullet$  $\mO_B(2p - 2 \tau) = \mO_B$. In this case
\begin{equation*}
h^0(G_1)=2, \quad h^1(G_1)=1, \quad h^0(G_2)=11,
\end{equation*}
hence $h^0(\widetilde{A}_6) \leq 10$ by \eqref{coker}. We have $1$
parameter for $B$, $2$ parameters for $\xi$, no parameters for
$\tau$ and $V_1$ and at most $9$
parameters from $\mathbb{P}H^0( \widetilde{A}_6)$. \\

Summing up, we conclude that $\mathcal{M}_{\textrm{IIa}}$ has
dimension at most $12$.
\end{proof}

\begin{proposition}\label{prop:IIb}
The stratum $\mathcal{M}_{\emph{IIb}}$ has dimension at most $12$.
\end{proposition}
\begin{proof}
Set $ W =  E_{\tau} (2, \, 1)$; then $ V_2(-2o) = W \oplus \mO_B$
and tensoring the exact sequence \eqref{A6} with $\mO_B(-6o)$ we
obtain
\begin{equation} \label{sequence:IIb}
0 \lr W \oplus  \mO_B \stackrel{i_3} \lr \big(\textrm{S}^3 W
\oplus \textrm{S}^2 W \big) \oplus ( W \oplus \mO_B)  \lr
A_6(-6o) \lr 0.
\end{equation}
Arguing as in \cite[Lemma 6.14]{CaPi06}, we see that the second
component of the map $i_3$ is actually the identity, hence the
exact sequence \eqref{sequence:IIb} splits, giving
\begin{equation*}
\widetilde{ A}_6= A_6(-4o -2 \tau)= (\textrm{S}^3 W \oplus
\textrm{S}^2 W ) (2o-2 \tau).
\end{equation*}
By Proposition \ref{prop:vec-ell} this in turn implies
\begin{equation*}
\widetilde{ A}_6 = \bigg( W \oplus W \oplus \bigoplus_{i=1}^3  L_i
\bigg)(2o- \tau),
\end{equation*}
hence $h^0(\widetilde{ A}_6)=9$. We have $1$ parameter for $B$, no
parameters for $V_1$, $2$ parameters for $\xi$, $1$ parameter for
$\tau$ and $8$ parameters from $\mathbb{P}H^0(\widetilde{ A}_6)$.
Therefore $\mathcal{M}_{\textrm{IIb}}$ has dimension at most $12$.

The fact that it is nonempty can be proven as in case
$\mathcal{M}_{\textrm{I}}$ (using an isogeny of degree $2$ instead
of $3$); the details are left to the reader.
\end{proof}

\begin{proposition} \label{VIb}
The stratum $\mathcal{M}_{\emph{IIc}}$ has dimension at most $11$.
\end{proposition}
\begin{proof}
In case (IIc) we have $\mO_B(2o-2p)= \mO_B(2p- 2o)$, with
$\mO_B(2o-2p) \neq \mO_B$, and the map $\sigma_2$ has the form
\begin{equation*}
\sigma_2 \colon \mO_B(2p) \oplus \mO_B(2o) \oplus \mO_B(4o-2p) \lr
\mO_B(2p) \oplus \mO_B(2o) \oplus \mO_B(4o - 2p + \tau).
\end{equation*}
Take global coordinates
\begin{equation*}
x_0 \in H^0 \big(\mO_{\mathbb{P}(V_1)}(1)  \otimes
\pi_1^*\mO_B(-p) \big), \quad x_1 \in H^0
\big(\mO_{\mathbb{P}(V_1)}(1)  \otimes \pi_1^*\mO_B(-2o+p) \big)
\end{equation*}
on the fibres of $\mathbb{P}(V_1)$ and, similarly, global
coordinates $y_0$, $y_1$, $y_2$ on the fibres of
$\mathbb{P}(V_2)$. With respect to these coordinates, $\sigma_2$
is given by
\begin{equation*}
\left\{
   \begin{array}{l}
     \sigma_2(x_0^2) = a_{00}y_0 + a_{02} y_2 \\
     \sigma_2(x_0x_1) = a_{11} y_1 + a_{12} y_2 \\
     \sigma_2(x_1^2)= a_{20}y_0 + a_{22} y_2,
   \end{array}
 \right.
 \end{equation*}
where $a_{00}, \, a_{11}, \, a_{20} \in \mathbb{C}$, $a_{02}, \,
a_{22} \in H^0(\mO_B(\tau))$, $a_{12} \in H^0(\mO_B(\tau + 2o - 2p
))$. Therefore the equation of the conic bundle
$\mathcal{C}\subset \mathbb{P}(V_2)$ is
\begin{equation*}
(a_{00}y_0 + a_{02} y_2)(a_{20}y_0 + a_{22} y_2)-(a_{11} y_1 +
a_{12} y_2)^2 =0.
\end{equation*}
Moreover, since the rank of $\sigma_2$ drops exactly at the point
$\tau$, it follows $a_{11} \neq 0$. This means that the
coefficient of the term $y_1^2$ is a non-zero constant, hence the
same argument of \cite[Lemma 3.5]{Pi09} shows that exact sequence
\eqref{A6} splits. Therefore we obtain
\begin{equation*} \label{A6-split-VIb}
\begin{split}
\widetilde{  A}_6 &=\mO_B(2p - 2 \tau) \oplus \mO_B(4o-2p +\tau)
\oplus \mO_B(4p- 2o- 2\tau) \\ &\oplus \mO_B(2p - \tau) \oplus
\mO_B(4o - 2p) \oplus \mO_B(6o-4p) \oplus \mO_B(2o- \tau),
\end{split}
\end{equation*}
so
\begin{equation*}
h^0(\widetilde{A}_6) = \left\{%
 \begin{array}{ll}
    10 & \hbox{if either $\mO(2p-2 \tau)=\mO_B$ or $\mO_B(4p - 2o - 2 \tau)=\mO_B$;} \\
    \;  9 & \hbox{otherwise.} \\
\end{array}%
\right.
\end{equation*}
So we have $1$ parameter for $B$, $1$ parameter for $\xi$, no
parameters (resp. $1$ parameter) for $\tau$ and $9$ parameters
(resp. $8$ parameters) from $\mathbb{P}H^0(\widetilde{A}_6)$. It
follows that $\mathcal{M}_{\textrm{IIc}}$ has dimension at most
$11$.
\end{proof}

\begin{proposition}
We have $\mathcal{M}_{\emph{IIIa}}=\mathcal{M}_{\emph{IIIb}}$.
Moreover the dimension of this stratum is at most $12$.
\end{proposition}
\begin{proof}
Case (IIIb) is obtained from case (IIIa) by considering $2o-p$
instead of $p$; this shows that the corresponding strata coincide.
So it is sufficient to consider case (IIIa); set
\begin{equation*}
 W :=E_{\tau + 2o - 2p}, \quad L:=\mO_B(2p-2o).
\end{equation*}
Then we have $ V_2(-2o)= W \oplus  L$ and tensoring the exact
sequence \eqref{A6} with $\mO_B(-6o)$ we obtain
\begin{equation*}
0 \lr W \oplus   L \stackrel{i_3} \lr \textrm{S}^3 W \oplus
(\textrm{S}^2 W \otimes  L) \oplus ( W \otimes  L^2) \oplus  L^3
\lr   A_6(-6o) \lr 0.
\end{equation*}
Hence $\widetilde{ A}_6 =  A_6(-4o-2 \tau)$ fits into the short
exact sequence
\begin{equation*}
0 \lr  G_1 \lr  G_2 \lr \widetilde{ A}_6 \lr 0,
\end{equation*}
where
\begin{equation*}
 G_1 = ( W \oplus  L)(2o-2 \tau), \quad  G_2 = \big(\textrm{S}^3 W \oplus (\textrm{S}^2 W \otimes  L)
\oplus ( W \otimes  L^2) \oplus  L^3 \big) (2o - 2 \tau).
\end{equation*}
There are several possibilities. \\  \\
$\bullet$ $ L(2o-2 \tau) \neq \mO_B, \,   L^3(2o-2 \tau) \neq
\mO_B$. In this case
\begin{equation*}
h^0( G_1)=1, \quad h^1( G_1)=0, \quad h^0( G_2)=10,
\end{equation*}
hence $h^0(\widetilde{A}_6)=9$. We have $1$ parameter for $B$, $1$
parameter for $V_1$, $1$ parameter for $\xi$, $1$ parameter for
$\tau$ and $8$
parameters from $\mathbb{P}H^0(\widetilde{  A}_6)$. \\ \\
$\bullet$ $ L(2o-2 \tau) \neq \mO_B, \,   L^3(2o- 2 \tau) =
\mO_B$. In this case
\begin{equation*}
h^0( G_1)=1, \quad h^1( G_1)=0, \quad h^0( G_2)=11,
\end{equation*}
hence $h^0(\widetilde{A}_6)=10$. We have $1$ parameter for $B$,
$1$ parameter for $V_1$, $1$ parameter for $\xi$, no parameters
for $\tau$ and $9$
parameters from $\mathbb{P}H^0(\widetilde{  A}_6)$. \\
\\
$\bullet$ $L(2o-2 \tau) = \mO_B, \,L^3(2o - 2 \tau) \neq \mO_B$.
We have
\begin{equation*}
h^0(G_1)=2, \quad h^1(G_1)=1, \quad h^0(G_2)=10,
\end{equation*}
hence $h^0(\widetilde{A}_6) \leq 9$ by \eqref{coker}. We have $1$
parameter for $B$, $1$ parameter for $V_1$, $1$ parameter for
$\xi$, no parameters for $\tau$ and at most $8$ parameters from
$\mathbb{P}H^0(\widetilde{  A}_6)$.
\\ \\
$\bullet$ $L(2o-2 \tau) = \mO_B, \,L^3(2o - 2 \tau) = \mO_B$.
Notice that this implies $L^2 = \mO_B$, so there are no parameters
for $V_1$.  We obtain
\begin{equation*}
h^0(G_1)=2, \quad h^1(G_1)=1, \quad h^0(G_2)=11,
\end{equation*}
hence $h^0(\widetilde{A}_6) \leq 10$ by \eqref{coker}. We have $1$ parameter for $B$, $1$ parameter for $\xi$, no parameters for $\tau$ and at most $9$ parameters from $\mathbb{P}H^0(\widetilde{  A}_6)$. \\ \\
Summing up, we conclude that the dimension of the stratum
$\mathcal{M}_{\textrm{IIIa}}=\mathcal{M}_{\textrm{IIIb}}$ is at
most $12$.
\end{proof}

\begin{proposition}
The stratum $\mathcal{M}_{\emph{IIIc}}$ has dimension at most
$12$.
\end{proposition}
\begin{proof}
As in the proof of \ref{prop:IIb}, $h^0(\widetilde{A}_6) = 9$. We
have $1$ parameter for $B$, $1$ parameters for $V_1$, $1$ parameter
for $\xi$, $1$ parameter for $\tau$ and $8$ parameters from
$\mathbb{P}H^0( \widetilde{ A}_6)$. Therefore
$\mathcal{M}_{\textrm{IIIc}}$ has dimension at most $12$.
\end{proof}

\begin{proposition}
The strata $\mathcal{M}_{\emph{IVa}}$, $\mathcal{M}_{\emph{IVb}}$
have dimension at most $11$.
\end{proposition}
\begin{proof}
The proof is the same as in case (IIc); the details are left to
the reader.
\end{proof}

\begin{proposition}
The stratum $\mathcal{M}_{\emph{IVc}}$ has dimension at most $11$.
\end{proposition}
\begin{proof}
In case (IVc) the vector bundles $G_1$, $G_2$ in exact sequence
\eqref{A6s} are as follows:
\begin{equation*}
G_1= \mO_B(2p - 2 \tau) \oplus \mO_B(4o - 2p - 2 \tau) \oplus
\mO_B(2o-\tau),
\end{equation*}
\begin{equation*}
\begin{split}
G_2 = & \mO_B(6p - 4o - 2 \tau) \oplus \mO_B(8o - 6p - 2 \tau) \oplus \mO_B(2o + \tau) \oplus \mO_B(2p - 2 \tau) \\
 & \oplus  \mO_B(4p - 2o - \tau) \oplus \mO_B(4o - 2p - 2 \tau) \oplus \mO_B(6o - 4p - \tau) \oplus \mO_B(2p) \\
 & \oplus  \mO_B(4o - 2p) \oplus \mO_B(2o - \tau).
\end{split}
\end{equation*}
A tedious but elementary analysis of all possibilities, together
with inequality \eqref{coker}, shows that the number of parameters
involved in the construction never exceeds $11$. Hence
$\mathcal{M}_{\textrm{IVc}}$ has dimension at most $11$.
\end{proof}

Now let us write $\mathcal{M}_{\textrm{V}} =
\mathcal{M}_{\textrm{V}, \, \textrm{gen}}\, \cup \,
\mathcal{M}_{\textrm{V}, 2}$, where $\mathcal{M}_{\textrm{V},
\,2}$ consists of surfaces with $\mO_B(2o-2 \tau)= \mO_B$ and
$\mathcal{M}_{\textrm{V}, \, \textrm{gen}}$ is the rest.

\begin{proposition} \label{prop.V.nonempty}
$\mathcal{M}_{\emph{V}, \, \emph{gen}}$ and
$\mathcal{M}_{\emph{V}, \,2}$ are both non-empty.
\end{proposition}
\begin{proof}
In case (V) we have $\mO_B(2o-2p)= \mO_B$, hence the map $\sigma_2
\colon S^2 V_1 \to V_2$ has the form
\begin{equation*}
\sigma_2 \colon \mO_B(2o)^3 \lr  \mO_B(2o)^2 \oplus \mO_B(2o +
\tau).
\end{equation*}
Recall that for a general choice of $\sigma_2$ we have $\bar{f}_i
\neq 0$ for all $i \in \{1, \, 2, \, 3\}$. Take coordinates $x_0,
\, x_1$ on the fibres of $V_1$ and $y_0, \, y_1, \, y_2$ on $V_2$;
with respect to these coordinates, $\sigma_2$ is given by
\begin{equation*}
\left\{
   \begin{array}{l}
     \sigma_2(x_0^2) = a_{00}y_0 + a_{01}y_1 + a_{02} f_0 y_2 \\
     \sigma_2(x_0x_1)= a_{10}y_0 + a_{11} y_1 + a_{12} f_0 y_2 \\
     \sigma_2(x_1^2)=  a_{20}y_0 +a_{21}y_1+ a_{22} f_0 y_2,
   \end{array}
 \right.
 \end{equation*}
where $a_{ij} \in \mathbb{C}$ and $f_0 \in H^0(\mO_B(\tau))$.
Moreover, since the rank of $\sigma_2$ drops precisely at the
point $\tau$, it follows $\det (a_{ij}) \neq 0$.

Therefore the global equation of the relative conic
$\mathcal{C}\subset \mathbb{P}(V_2)$ is
\begin{equation*}
(a_{00} y_0 + a_{01}y_1 + a_{02} f_0 y_2)(a_{20} y_0 + a_{21}y_1 +
a_{22} f_0y_2)-(a_{10} y_0 + a_{11}y_1 + a_{12} f_0 y_2)^2=0.
\end{equation*}
Notice that at least one of the coefficient of $y_0^2$, $y_1^2$ or
$y_0 y_1$ in the equation of $\mathcal{C}$ is not zero, otherwise
$y_2^2$ divides the equation of $\mathcal{C}$. Since each of these
coefficients is a non-zero constant, by the argument in
\cite[Lemma 3.5]{Pi09} one sees that in any case the exact
sequence \eqref{A6} splits. Therefore we obtain
\begin{equation*} \label{A6-split-VII}
\widetilde{  A}_6 =\mO_B(2o - 2 \tau)^2 \oplus \mO_B(2o-\tau)^2
\oplus \mO_B(2o)^2 \oplus \mO_B(2o+ \tau),
\end{equation*}
so
\begin{equation*}
h^0(\widetilde{  A}_6) = \left\{%
 \begin{array}{ll}
    11 \; & \hbox{if $\mO_B(2o-2 \tau)=\mO_B$,} \\
    \;  9 & \hbox{otherwise.} \\
\end{array}%
\right.
\end{equation*}
Choosing $a_{02}=a_{22}=a_{10}=a_{11}=0$, $a_{00}=a_{01}=a_{20}=
a_{12}=1$, $a_{21}=-1$, the equation of $\mathcal{C}$ becomes
\begin{equation*}
y_0^2-y_1^2-f_0^2y_2^2=0.
\end{equation*}
Hence $\mathcal{C}$ has a unique singular point (of type $A_1$),
namely the point  $P$ with homogeneous coordinates $[0:0:1]$ lying
on the fibre over $\tau$; in particular, condition
$(\mathcal{P}_1)$ of Theorem \ref{teo:genus 2} is satisfied. Since
\eqref{A6} splits, the curve $\Delta$ defined by the section $w
\in H^0(\widetilde{A}_6)$ is cut by a relative cubic $\mathcal{G}
\in |\mO_{\mathbb{P}(V_2)} (3) \otimes \pi_2^* \mO_B(-4o-2
\tau)|$; let us write the equation of $\mathcal{G}$ as
\begin{equation} \label{eq-cubic}
\sum_{i+j+k=3}b_{ijk}\;y_0^iy_1^jy_2^k=0,
\end{equation}
where $b_{ijk} \in H^0(\mathbb{P}(V_2), \, \pi_2^{\ast}\mO_B(2o
+(k-2) \tau) )$. If $\mO_B(2o-2 \tau) = \mO_B$ then all the
coefficients of $\mathcal{G}$ are generically non-zero; one checks
that in this case the linear system $|\mathcal{G}|$ in
$\mathbb{P}(V_2)$ is base-point free, hence the linear system
$|\Delta|$ in $\mathcal{C}$ is base-point free too; by Bertini
theorem, we conclude that for a general choice of $\Delta$
condition $(\mathcal{P}_2)$ is also satisfied, hence
$\mathcal{M}_{\textrm{V}, \, 2}$ is non-empty.

If $\mO_B(2o-2 \tau) \neq \mO_B$, then
$b_{300}=b_{210}=b_{120}=b_{030}=0$. So the relative cubic
$\mathcal{G}$ splits as $\mathcal{G}=\mathcal{H} \cup
\mathcal{G}'$, where $\mathcal{H}$ is the relative hyperplane
$\{y_2=0 \}$ and $\mathcal{G}'$ is the relative conic
\begin{equation*}
b_{201}y_0^2+ b_{111}y_0y_1+
b_{102}y_0y_2+b_{021}y_1^2+b_{012}y_1y_2+b_{003}y_2^2=0.
\end{equation*}
Consequently, $\Delta$ splits as $\Delta =
\mathcal{H}_{\mathcal{C}} \cup \Delta'$, where
$\mathcal{H}_{\mathcal{C}}= \mathcal{H} \cap \mathcal{C}$ and
$\Delta' = \mathcal{G}' \cap \mathcal{C}$. The sections $b_{201},
\, b_{021}, \, b_{111}$ all vanish at the same point, namely the
unique point $q \in B$ such that $\mO_B(2o- \tau)= \mO_B(q)$;
notice that $q \neq \tau$. Hence the base locus of
$|\mathcal{G}'|$ is the line $y_2=0$ in the fibre $\pi^{-1}(q)$,
and this in turn implies that the base locus of $|\Delta'|$ in
$\mathcal{C}$ are the two points $P_1=[1: \, 1: \, 0]$ and
$P_2=[1: \, -1: \, 0]$ on the fibre of $\mathcal{C}$ over $q$. Now
let us make a general choice of the coefficients in
\eqref{eq-cubic}. Then
  $\Delta$ does not contain the unique singular point of $\mathcal{C}$;
  moreover, a standard local computation together with Bertini theorem show that
\begin{itemize}
\item $\Delta'$ is smooth; \item $\Delta'$ and
$\mathcal{H}_{\mathcal{C}}$ intersect transversally at $P_1$ and
$P_2$.
\end{itemize}
So condition $(\mathcal{P}_2)$ is satisfied and
$\mathcal{M}_{\textrm{V}, \, \textrm{gen}}$ is non-empty.
\end{proof}

Let us compute now the dimensions of $\mathcal{M}_{V, \, 2}$ and
$\mathcal{M}_{V, \, \textrm{gen}}$.

\begin{proposition}
$\mathcal{M}_{\emph{V}, \,2}$ has dimension $12$, whereas
$\mathcal{M}_{\emph{V}, \, \emph{gen}}$ has dimension $11$.
Moreover, $\mathcal{M}_{\emph{V}, \,2}$ is a generically smooth,
irreducible component of $\mathcal{M}$.
\end{proposition}
\begin{proof}
We first compute the dimension of the parameter space
$\mathcal{D}$ in each case. If $\mO_B(2o-2 \tau)=\mO_B$ we have
$1$ parameter for $B$, $2$ parameters for $\xi$ and $10$
parameters from $\mathbb{P}H^0(\widetilde{  A}_6)$; otherwise we
have $1$ parameter for $B$, $2$ parameters for $\xi$, $1$
parameter from $\tau$ and $8$ parameters from
$\mathbb{P}H^0(\widetilde{A}_6)$. Therefore
$\mathcal{M}_{\emph{V}, \,2}$ has dimension at most $13$, whereas
$\mathcal{M}_{\emph{V}, \, \textrm{gen}}$ has dimension at most
$12$.

By Remark \ref{rem.3.2}, we have now
 to find the dimension of the general fibre of $\Phi \colon
\mathcal{D} \to \mathcal{M}$, and for this we have to consider the
action of certain automorphism groups over our data.

Observe first that in both cases we can forget the action of
$\textrm{Aut}(B)$, since we have fixed a point of $B$ by choosing
$\det V_1=\mO_B(2o)$. So we have only to consider the action of
$\textrm{Aut}(V_1) \times \textrm{Aut}(V_2)$.

We are therefore reduced to solve the following problem: given an
admissible $5$-tuple $(B, \, V_1, \, t, \, \xi, \, w)$,
corresponding to the genus $2$ fibration $f \colon S \to B$, we
must find the dimension of the subvariety $Z \subset
\textrm{Aut}(V_1) \times \textrm{Aut}(V_2)$ given by the pairs
$(\phi_1, \, \phi_2)$ which make the following diagram commuting:

\begin{equation} \label{dia.V}
\xymatrix{ 0 \ar[r] & S^2 V_1 \ar[r]^{\sigma_2} \ar[d]_{S^2
\phi_1} & V_2
\ar[r] \ar[d]^{\phi_2} & \mO_{\tau} \ar[r] \ar@{=}[d]  & 0\\
0 \ar[r] & S^2 V_1 \ar[r]^{\sigma_2}  & V_2 \ar[r] & \mO_{\tau}
\ar[r]  & 0.}
\end{equation}

In fact, the dimension of the the fibre $\Phi^{-1}({[S]})$ is
given by $\dim Z - 1$. Geometrically, this expresses the fact that
the points in such a fibre are in $1$-to-$1$ correspondence with
the family of automorphisms of the projective bundle
$\mathbb{P}(V_2)$ fixing the conic bundle $\mathcal{C}$.

Now we claim that, if $S$ is \emph{general} in either
$\mathcal{M}_{\emph{V}, \,2}$ or $\mathcal{M}_{\emph{V}, \,
\textrm{gen}}$, by choosing suitable coordinates for $V_1$ and
$V_2$ we can put the equation of the conic bundle $\mathcal{C}$ in
the form
\begin{equation} \label{eq.conic.gen}
y_0(y_1+f_0y_2)-y_1^2=0.
\end{equation}
In fact, in the general case $\mathcal{C}$ has a nodal fibre over
the point $\tau$; without loss of generality we can assume that
such a fibre has equation $y_1(y_0-y_1)=0$, so that the conic
bundle has the form
$(y_0+a_{02}f_0y_2)(y_1+a_{22}f_0y_2)-(y_1+a_{12}f_0y_2)^2=0$. Now
 the claim follows by using the linear change of coordinates
 \begin{equation*}
y_0' :=y_0+a_{02}f_0y_2, \quad y_1':=y_1+a_{12}f_0y_2, \quad
y_2':=(a_{22}-a_{12})y_2.
\end{equation*}

Therefore, in order to compute the dimension of the general fibre
of $\Phi$, we may assume that the matrix associated with $\sigma_2
\colon S^2V_1 \to V_2$ is
\begin{equation*}
\left(
        \begin{array}{ccc}
            1 & 0 & 0 \\
            0 & 1 & 1 \\
            0 & 0 & f_0 \\
        \end{array}
         \right).
\end{equation*}

Let now $\phi_1 \in \textrm{Aut}(V_1)$, given by $\phi_1 (x_0) = a
x_0 + c x_1$ and $\phi_1 (x_1) = b x_0 + d x_1$, $a,\, b, \, c, \,
d \in \mathbb{C}$. Then the action of $S^2 \phi_1$ on $S^2V_1$ is
expressed by the matrix
\begin{equation*}
        \left(
        \begin{array}{ccc}
            a^2 & ab & b^2 \\
            2ac & ad + bc & 2bd \\
            c^2 & cd & d^2 \\
        \end{array}
         \right).
\end{equation*}
On the other hand, the general $\phi_2 \in \textrm{Aut}(V_2)$ is
given by
\begin{equation*}
    \left(
   \begin{array}{ccc}
     b_{11} & b_{12} & 0 \\
     b_{21} & b_{22} & 0 \\
     b_{31} f_0 & b_{32} f_0 & b_{33} \\
   \end{array}
    \right),
\end{equation*}
where $b_{ij} \in \mathbb{C}$. Hence, imposing that the diagram
\eqref{dia.V} commutes, by straightforward computations one finds
that any pair $(\phi_1, \, \phi_2) \in Z$ is either of the form
\begin{equation*}
\phi_1= \left(
\begin{array}{cc}
a & a \\
c & -a \\
\end{array}
\right), \quad \phi_2 = \left(
  \begin{array}{ccc}
    a^2 & a^2 & 0 \\
     2ac+c^2 & -a^2 & 0 \\
     c^2 f_0 & -ac f_0 & a(a+c) \\
   \end{array}
   \right)
\end{equation*}
or of the form
\begin{equation*}
\phi_1= \left(
\begin{array}{cc}
a & 0 \\
c & a+c \\
\end{array}
\right), \quad \phi_2 = \left(
  \begin{array}{ccc}
    a^2 & 0 & 0 \\
     2ac+c^2 & (a+c)^2 & 0 \\
     c^2 f_0 & c(a+c) f_0 & a(a+c) \\
   \end{array}
   \right).
\end{equation*}
It follows that $Z \subset \textrm{Aut}(V_1) \times
\textrm{Aut}(V_2)$ is a subvariety of dimension $2$. Consequently,
the general fibre of $\Phi$ has dimension $1$; this means that the
dimension of $\mathcal{M}_{V, \, 2}$ equals $12$, whereas the
dimension of
$\mathcal{M}_{V, \, \textrm{gen}}$ equals $11$. \\

Now we want to prove that $\mathcal{M}_{V, \, 2}$ is an
irreducible component of $\mathcal{M}$. In order to do this, we
will show that $h^1(S, \, T_S)=12$ for a general $S \in
\mathcal{M}_{V, \, 2}$. Since $\dim\ \mathcal{M}_{V, \, 2} = 12$,
this will also prove that this component is generically smooth.

The condition $h^1 (S, \, T_S) \leq 12$ is equivalent to $h^2(S,
\, T_S) = h^0(S, \, \Omega_S^1 \otimes \omega_S) \leq 2$. By
Remark \ref{rem.3.2}, it is therefore enough to prove that $h^0
(\mathcal{F}) = 0$, where $\mathcal{F} := (\Omega_S^1 \otimes
\omega_S) / \omega_S$ or, equivalently, that there are no
bicanonical curves of $S$ containing the $0$--dimensional scheme
$\textrm{Crit} (f)$.

By the results in Subsection \ref{C-P}, the Albanese fibration $f
\colon S \to B$ factors as the composition of the conic bundle
$\mathcal{C} \lr B$ and a finite double cover $\psi:S \lr
\mathcal{C}$ branched on the node of $\mathcal{C}$ and on a smooth
curve $\Delta$ not passing through the node.

Let us study the $0$-dimensional scheme $\textrm{Crit}(f)$. Since
all the fibres of $\mathcal{C}$ are reduced, the critical points
of $f$ must be fixed by the involution of $S$. The isolated fixed
point is the preimage of the node of $\mathcal{C}$, and it is
critical for $f$. The other critical points of $f$ are the points
of $S$ whose images in $\mathcal{C}$ are the ramification points
for the map $\Delta \lr B$.

As before, we can choose $\mathcal{C}$ of equation $y_0^2 - y_1^2
+ f_0y_2^2 = 0$, and the curve $\Delta$ is defined as the complete
intersection of $\mathcal{C}$ with a relative cubic $\mathcal{G}
\in |\mO_{\mathcal{C}}(3) \otimes \mO_B(-4o - 2\tau)|$. Since
$\mathcal{O}_B(2o-2 \tau) = \mathcal{O}_B$, we can choose
$\mathcal{G}$ of equation
\begin{equation*}
                    a y_0^3 + b y_1^3 + \lambda y_2^3 = 0,
\end{equation*}
where $a, \, b \in \mathbb{C}$ and $\lambda \in
H^0(\mathbb{P}(V_2), \, \pi_2^*\mO_B(3\tau))$, see
\eqref{eq-cubic}. The node $P$ of $\mathcal{C}$ is the point with
homogeneous coordinates $[0:0:1]$ lying on the fibre over $\tau$,
and $\textrm{Crit} (\Delta \lr B)$ is defined by
\begin{equation*}
\textrm{rank} \left (
    \begin{array}{ccc}
      y_0 & -y_1 & f_0^2 y_2 \\
      a y_0^2 & b y_1^2 & \lambda y_2^2
    \end{array} \right ) \leq 1.
\end{equation*}
This is obviously equivalent to set equal $0$ all the minors of
order $2$. So we must solve the system of equations
\begin{equation*} \label{crit}
\left\{
   \begin{array}{l}
     b y_0y_1^2 + a y_0^2y_1 = 0 \\
     \lambda y_0y_2^2 - a f_0^2 y_0^2y_2 = 0 \\
     \lambda y_1y_2^2 + b y_1^2y_2 = 0,
   \end{array}
\right. \quad \textrm{that is} \quad \left\{
   \begin{array}{l}
     y_0y_1(b y_1 + a y_0) = 0 \\
     y_0y_2(\lambda y_2 - a f_0^2 y_0) = 0 \\
     y_1y_2(\lambda y_2 + b f_0^2 y_1) = 0.
   \end{array}
\right.
\end{equation*}

This yields
\begin{equation*}
\begin{split}
&\{y_0 = y_1 = 0\} \: \cup \: \{y_0 = y_2 = 0\}   \: \cup  \:  \{y_0 = \lambda y_2 + b f_0^2 y_1 = 0\}   \: \cup  \:  \\ &\{y_1 = y_2 = 0\}   \: \cup  \:  \{y_1 = \lambda y_2 - a f_0^2 y_0 = 0\}   \: \cup  \:  \\
&\{y_2 = b y_1 + a y_0 = 0\} \: \cup \: \{\lambda y_2 + b f_0^2
y_1 = \lambda y_2 - a f_0^2 y_0 = b y_1 + a y_0 = 0\}.
\end{split}
\end{equation*}
 Let us compute, in each case, the solutions in $\mathcal{C}$:
 \begin{description}
   \item[$\{y_0 = y_1 = 0\}$] In this case, because $f_0(\tau) = 0$, the unique solution in $\mathcal{C}$ is the point $P$.
   \item[$\{y_0 = y_2 = 0\}$] By looking at the equation of $\mathcal{C}$ we have also that $y_1 = 0$, and this is impossible. So in this case there are no solutions.
   \item[$\{y_0 = \lambda y_2 + b f_0^2 y_1 = 0\}$] We must solve
\begin{equation*}
   \left\{
   \begin{array}{l}
     y_0 = \lambda y_2 + b f_0^2 y_1 = 0 \\
     y_0^2 - y_1^2 + f_0^2y_2^2 = 0,
\end{array}
\right. \quad \textrm{that is} \quad
   \left\{
   \begin{array}{l}
     y_0 = 0 \\
     \lambda y_2 + b f_0^2 y_1 = 0 \\
     (- y_1 + f_0 y_2) (y_1 + f_0 y_2) = 0,
   \end{array}
\right.
\end{equation*}
which gives
\begin{equation*}
   \left\{
   \begin{array}{l}
     y_0 = 0 \\
     y_1 = f_0 y_2 \\
     y_2 (\lambda + b f_0^3) = 0
   \end{array}
\right. \quad \cup \quad \left\{
                           \begin{array}{l}
                             y_0 = 0 \\
                             y_1 = - f_0 y_2 \\
                             y_2 (\lambda - b f_0^3) = 0.
                           \end{array}
\right.
\end{equation*}
Since $y_2 \neq 0$ the solutions are the three points
$[0:f_0(\rho_i):1]$ lying on the fibres over $\rho_i$, where
$\rho_1 + \rho_2 + \rho_3 = \textrm{div} (\lambda + b f_0^3)$, and
the three points $[0: -f_0(\rho'_i):1]$ lying on the fibres over
$\rho'_i$, where $\rho'_1 + \rho'_2 + \rho'_3 = \textrm{div}
(\lambda - b f_0^3)$.
   \item[$\{y_1 = y_2 = 0\}$] The equation of $\mathcal{C}$ also gives $y_0 = 0$, which is impossible; so in this case there are no solutions.
   \item[$\{y_1 = \lambda y_2 - a f_0^2 y_0 = 0\}$] The computations are the same as in the case $\{y_0 = \lambda y_2 + b f_0^2 y_1 = 0\}$.
   The solutions are the three points $[-\sqrt{-1}f_0(\varepsilon_i):0:1]$ lying on the fibres over $\varepsilon_i$, where
   $\varepsilon_1 + \varepsilon_2 + \varepsilon_3 = \textrm{div} (\lambda + a \sqrt{-1} f_0^3)$, and the three points $[\sqrt{-1}f_0(\varepsilon_i):0:1]$ lying on the fibres over
   $\varepsilon'_i$,
    where $\varepsilon'_1 + \varepsilon'_2 + \varepsilon'_3 = \textrm{div} (\lambda - a \sqrt{-1} f_0^3)$.
   \item[$\{y_2 = b y_1 + a y_0 = 0\}$] From the equation of $\mathcal{C}$ it follows that for a generic choice of $a$ and $b$ we must have $y_0 = y_1 = y_2 = 0$, which is impossible.
   So in this case there are no solutions.
   \item[$\{\lambda y_2 + b f_0^2 y_1 = \lambda y_2 - a f_0^2 y_0 = b y_1 + a y_0 = 0\}$] In this case
we find six points, three on the curve
       \begin{equation*}
       \left\{
       \begin{array}{l}
         b y_1 + a y_0 = 0 \\
         c y_0 + b f_0 y_2 = 0
       \end{array}
       \right.
       \end{equation*}
       and three on the curve
       \begin{equation*}
       \left\{
       \begin{array}{l}
         b y_1 + a y_0 = 0 \\
         c y_0 - b f_0 y_2 = 0,
       \end{array}
       \right.
       \end{equation*}
where $- c^2 = b^2 - a^2$. In general, $a, \, b$ and $c$ are
nonzero and, in such a case, the solutions are the three points
$[-b f_0(\sigma_i):a f_0(\sigma_i):c]$ lying on the fibres over
$\sigma_i$, where $\sigma_1 + \sigma_2 + \sigma_3 = \textrm{div}
(c \lambda + a b f_0^3)$ and the three points $[b
f_0(\sigma'_i):-a f_0(\sigma'_i):c]$ lying on the fibres over
$\sigma'_i$, where $\sigma'_1 + \sigma'_2 + \sigma'_3 =
\textrm{div} (c \lambda - a b f_0^3)$.
\end{description}
Summing up, for a general $S \in \mathcal{M}_{\textrm{V}, 2}$ the
$0$-dimensional scheme $\textrm{Crit}(f)$ consists precisely of
$19$ distinct points. One is the preimage $Q := \psi^{-1}(P)$ of
$P$ in $S$, and the others correspond to the singularities of
eighteen $2$-connected nodal curves, as in the following picture:
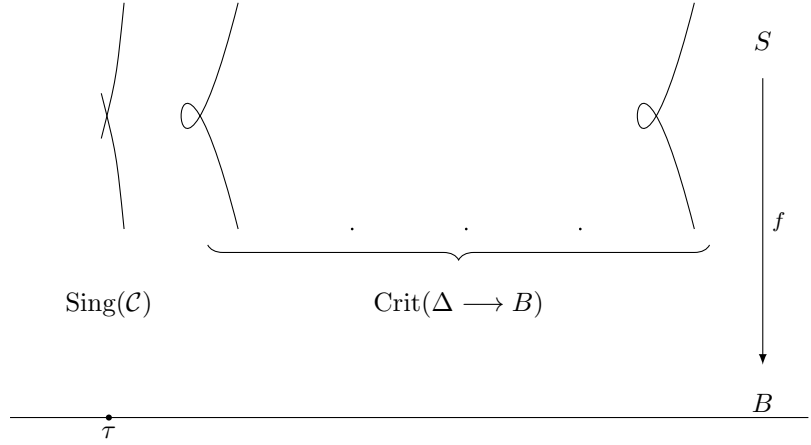
\begin{center}
\begin{tikzpicture}[point/.style={circle,inner sep=0pt,minimum size=1pt,fill=black},
skip loop/.style={to path={-- ++(0,#1) -| (\tikztotarget)}}]
 \node (p1) [point ] at (3,0) [] {};
 \node (p2) [point] at (4.5,0) [] {};
 \node (p3) [point] at (6,0) [] {};
 \node (Crit) at (4.4,-1) [] {$\textrm{Crit}(\Delta \lr B)$};
 \node (Sing) at (-0.2,-1) [] {$\textrm{Sing}(\mathcal{C})$};

 \node (S) at (8.4,2.5) [] {$S$};
 \node (B) at (8.4,-2.3) [] {$B$};

 \fill (-0.2,-2.5) circle (0.04cm);
 \node (tau) at (-0.2,-2.5) [anchor=north] {$\tau$};

 \draw[-latex] (8.4,2) -- node[right]{\scriptsize{$f$}} (8.4,-1.8);

 \draw (0,0) .. controls (-0.1,1) .. (-0.3,1.8);
 \draw (0,3) .. controls (-0.1,2) .. (-0.3,1.2);
 \draw (1.5,0) .. controls (0.5,4) and (0.5,-1) .. (1.5,3);
 \draw (7.5,0) .. controls (6.5,4) and (6.5,-1) .. (7.5,3);

 \draw[decorate,decoration={brace,mirror,amplitude=2mm,raise=6pt}] (1.1,0) -- (7.7,0);

 \draw (-1.5,-2.5) -- (9,-2.5);
\end{tikzpicture}
\end{center}

Notice that this agrees with the Zeuthen--Segre formula
\begin{equation*}
19 = \chi_{\textrm{top}}(S) =
\chi_{\textrm{top}}(B)\chi_{\textrm{top}}(F) + \sum
\chi_{\textrm{top}}(F_p) - \chi_{\textrm{top}}(F) = \sum
\chi_{\textrm{top}}(F_p) - \chi_{\textrm{top}}(F),
\end{equation*}
where the sum runs over the singular fibres of $f$. Thus for a
general $S \in \mathcal{M}_{V, \, 2}$, the Albanese map has
exactly $19$ singular fibres.

Since the linear system $|2K_S|$ is the pullback via the relative
bicanonical map of the linear system $|\mO_{\mathbb{P}(V_2)}(1)|$,
we must now compute the dimension of the vector space of elements
in $H^0(\mO_{\mathbb{P}(V_2)}(1))$ which contain
$\textrm{Crit}(f)$.

Let us consider the six curves
\begin{equation*}
A_1\colon \left\{
       \begin{array}{l}
         y_0 = 0 \\
         y_1 - f_0 y_2 = 0,
       \end{array}
    \right.         \:  \:        A_2\colon \left\{
                                        \begin{array}{l}
                                          y_0 = 0 \\
                                          y_1 + f_0 y_2 = 0,
                                        \end{array}
                                      \right.\\
\end{equation*}
\begin{equation*}
B_1\colon \left\{
       \begin{array}{l}
         y_1 = 0 \\
         y_0 - \sqrt{-1} f_0 y_2 = 0,
       \end{array}
    \right.         \:  \:        B_2\colon \left\{
                                        \begin{array}{l}
                                          y_0 = 0 \\
                                          y_0 + \sqrt{-1} f_0 y_2 = 0,
                                        \end{array}
                                      \right.\\
\end{equation*}
\begin{equation*}
C_1\colon \left\{
       \begin{array}{l}
         b y_1 + a y_0 = 0 \\
         c y_0 + b f_0 y_2 = 0,
       \end{array}
    \right.         \:  \:        C_2\colon \left\{
                                        \begin{array}{l}
                                          b y_1 + a y_0 = 0 \\
                                          c y_0 - b f_0 y_2 = 0.
                                        \end{array}
                                      \right.
\end{equation*}
Each curve contains $Q$ and three other points of $\textrm{Crit}
(f)$ as in the following picture:

\vskip 0.5cm

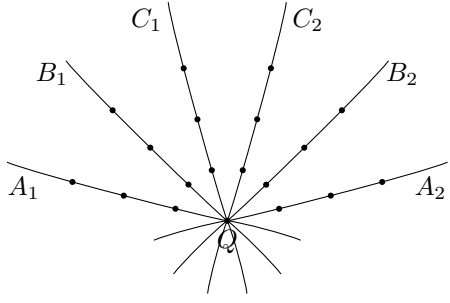
\begin{center}
\begin{tikzpicture}
\node (Q) at (0,0) [anchor=north] {$Q$};

\draw[rotate=15] (-1,0) .. controls +(8:0.2cm) and +(172:0.2cm) ..
(0,0); \draw[name path=A2, rotate=15] (0,0) .. controls
+(352:0.2cm) and +(188:0.4cm) .. node[very near end,below] {$A_2$}
(3,0);

\draw[rotate=45] (-1,0) .. controls +(8:0.2cm) and +(172:0.2cm) ..
(0,0); \draw[name path=B2, rotate=45] (0,0) .. controls
+(352:0.2cm) and +(188:0.4cm) .. node[very near end,right] {$B_2$}
(3,0);

\draw[rotate=75] (-1,0) .. controls +(8:0.2cm) and +(172:0.2cm) ..
(0,0); \draw[name path=C2, rotate=75] (0,0) .. controls
+(352:0.2cm) and +(188:0.4cm) .. node[very near end,right] {$C_2$}
(3,0);

\draw[rotate=105] (-1,0) .. controls +(352:0.2cm) and +(188:0.2cm)
.. (0,0); \draw[name path=C1, rotate=105] (0,0) .. controls
+(8:0.2cm) and +(172:0.4cm) .. node[very near end,left]
{$C_1$}(3,0);

\draw[rotate=135] (-1,0) .. controls +(352:0.2cm) and +(188:0.2cm)
.. (0,0); \draw[name path=B1, rotate=135] (0,0) .. controls
+(8:0.2cm) and +(172:0.4cm) .. node[very near end,left]
{$B_1$}(3,0);

\draw[rotate=165] (-1,0) .. controls +(352:0.2cm) and +(188:0.2cm)
.. (0,0); \draw[name path=A1, rotate=165] (0,0) .. controls
+(8:0.2cm) and +(172:0.4cm) .. node[very near end,below]
{$A_1$}(3,0);

\path[name path=circle1] (0,0) circle (0.7cm); \path[name
path=circle2] (0,0) circle (1.4cm); \path[name path=circle3] (0,0)
circle (2.1cm);

\fill (0,0) circle (0.04cm);

\fill [name intersections={of=A2 and circle1}] (intersection-1)
circle (0.04cm); \fill [name intersections={of=B2 and circle1}]
(intersection-1) circle (0.04cm); \fill [name intersections={of=C2
and circle1}] (intersection-1) circle (0.04cm); \fill [name
intersections={of=C1 and circle1}] (intersection-1) circle
(0.04cm); \fill [name intersections={of=B1 and circle1}]
(intersection-1) circle (0.04cm); \fill [name intersections={of=A1
and circle1}] (intersection-1) circle (0.04cm);

\fill [name intersections={of=A2 and circle2}] (intersection-1)
circle (0.04cm); \fill [name intersections={of=B2 and circle2}]
(intersection-1) circle (0.04cm); \fill [name intersections={of=C2
and circle2}] (intersection-1) circle (0.04cm); \fill [name
intersections={of=C1 and circle2}] (intersection-1) circle
(0.04cm); \fill [name intersections={of=B1 and circle2}]
(intersection-1) circle (0.04cm); \fill [name intersections={of=A1
and circle2}] (intersection-1) circle (0.04cm);

\fill [name intersections={of=A2 and circle3}] (intersection-1)
circle (0.04cm); \fill [name intersections={of=B2 and circle3}]
(intersection-1) circle (0.04cm); \fill [name intersections={of=C2
and circle3}] (intersection-1) circle (0.04cm); \fill [name
intersections={of=C1 and circle3}] (intersection-1) circle
(0.04cm); \fill [name intersections={of=B1 and circle3}]
(intersection-1) circle (0.04cm); \fill [name intersections={of=A1
and circle3}] (intersection-1) circle (0.04cm);
\end{tikzpicture}
\end{center}

The N\'{e}ron--Severi group $\textrm{NS}(\mathbb{P}(V_2))$ is
generated by $H$ and $\Psi$, where $H$ is the class of
$\mO_{\mathbb{P}(V_2)} (1)$ and $\Psi$ is the class of a fibre.

Let $Y$ be an element of $|\mO_{\mathbb{P}(V_2)} (1)|$ containing
$\textrm{Crit}(f)$. Thus $Y$ contains $4$ points in each curve
$A_j$, $B_j$, $C_j$, $j = 1, \, 2$. Since the numerical class of
these curves is $(H - 2\Psi)^2$, we have
\begin{equation*}
H(H - 2 \Psi)^2 = H(H^2 - 2 H \Psi) = H^3 - 4 H \Psi = 7 - 4 = 3
\end{equation*}
and so, by B\'{e}zout theorem, $Y$ contains all the curves $A_j$,
$B_j$, $C_j$. Let us write the equation of $Y$ as $\alpha y_0 +
\beta y_1 + \gamma y_2 = 0$, where $\alpha, \beta \in H^0(\pi_2^*
\mO_B (2o))$ and $\gamma \in H^0(\pi_2^* \mO_B (2o + \tau))$.

By imposing that $Y$ contains $A_1$, we find
\begin{equation*}
\beta f_0 y_2 + \gamma y_2 \equiv 0,
\end{equation*}
which implies $\gamma = -\beta f_0$. By imposing that $Y$ contains
$A_2$, we find
\begin{equation*}
-\beta f_0 y_2 + \gamma y_2 \equiv 0,
\end{equation*}
which implies $\gamma = \beta f_0$. It follows $\gamma = \beta =
0$, hence $Y$ has equation $\alpha y_0 = 0$. Similarly, by
imposing that $Y$ contains both $B_1$ and $B_2$, we obtain that
$Y$ is of the form $\beta y_1 = 0$. Thus $Y \equiv 0$, i.e.
\begin{equation*}
\textrm{Ker}[H^0(\omega_S^{\otimes 2}) \lr
H^0(\mO_{\textrm{Crit}(f)}(\omega_S^{\otimes 2}))] = 0,
\end{equation*}
which implies $h^1(T_S) = 12$. This shows that
$\mathcal{M}_{\textrm{V}, 2}$ is a generically smooth, irreducible
component of $\mathcal{M}$ of dimension $12$.
\end{proof}

Finally, we consider the strata belonging to $\mathcal{M}''$. The
surfaces belonging to these strata satisfy $V_1 = F_2 (\eta)$,
where $\eta$ is a $2$--torsion point, hence $V_1$ will not play
any role in the computation of parameters.

\begin{proposition}
The stratum $\mathcal{M}_{\emph{VI}}$ has dimension at most $12$.
\end{proposition}
\begin{proof}
Set $ W :=  E_{\tau}(3, \, 1)$; then we have a short exact
sequence
\begin{equation*}
0 \lr W(2o - 2 \tau) \lr \textrm{S}^3 W (2o - 2 \tau) \lr
\widetilde{ A}_6 \lr 0.
\end{equation*}
By \cite[Section 1]{CaCi93} we obtain
\begin{equation*}
h^0( W(2o - 2 \tau))=1, \quad  h^1( W(2o - 2 \tau))=0, \; \;
h^0(\textrm{S}^3 W(2o - 2 \tau))=10,
\end{equation*}
hence $h^0(\widetilde{ A}_6)=9$. We have $1$ parameter for $B$,
$2$ parameters for $\xi$, $1$ parameter for $\tau$
 and $8$ parameters from $\mathbb{P}H^0(\widetilde{  A}_6)$. Therefore
either $\mathcal{M}_{\textrm{VI}}$ has dimension at most $12$.
\end{proof}

\begin{proposition}
The stratum $\mathcal{M}_{\emph{VIIa}}$ has dimension at most
$11$.
\end{proposition}
\begin{proof}
In this case $V_2(-2o) = F_2 \oplus \mO_B(\tau)$, and $\xi$
belongs to a family which is at most $1$-dimensional, see
Proposition \ref{V2-indecomp}. The vector bundle $\widetilde{A}_6$
fits into a short exact sequence
\begin{equation*}
0 \lr G_1 \lr G_2 \lr \widetilde{A}_6 \lr 0,
\end{equation*}
where
\begin{equation*}
G_1 = \big(F_2 \oplus \mO_B(\tau) \big)(2o- 2 \tau), \quad G_2 =
\big(F_4 \oplus F_3(\tau) \oplus F_2(2 \tau) \oplus \mO_B(3 \tau)
\big) (2o - 2 \tau).
\end{equation*}
We distinguish two cases. \\
$(i)$ $\mO_B(2o - 2 \tau) \neq \mO_B$. We obtain
\begin{equation*}
h^0(G_1)=1, \quad h^1(G_1)=0, \quad h^0(G_2)=10,
\end{equation*}
therefore $h^0(\widetilde{A}_6)=9$. We have $1$ parameter for $B$,
at most one parameter for $\xi$, one parameter for $\tau$ and $8$
parameters from $\mathbb{P}H^0(
\widetilde{ A}_6)$. \\
$(ii)$ $\mO_B(2o - 2 \tau) = \mO_B$. We obtain
\begin{equation*}
h^0(G_1)=2, \quad h^1(G_1)=1, \quad h^0(G_2)=11,
\end{equation*}
 hence $h^0(\widetilde{A}_6) \leq 10$, see \eqref{coker}. We have $1$ parameter for $B$, at most one parameter for $\xi$,
no parameters for $\tau$ and at most $9$ parameters from
$\mathbb{P}H^0( \widetilde{ A}_6)$.

It follows that $\mathcal{M}_{\textrm{VIIa}}$ has dimension at
most $11$.
\end{proof}

\begin{proposition}
The stratum $\mathcal{M}_{\emph{VIIb}}$ has dimension at most
$11$.
\end{proposition}
\begin{proof}
In this case $\xi$ belongs to a family which is at most
$1$-dimensional. Set $ W =  E_{\tau} (2, \, 1)$; then $ V_2(-2o) =
W \oplus \mO_B$ and tensoring the exact sequence \eqref{A6} with
$\mO_B(-4o - 2 \tau)$ we obtain
\begin{equation} \label{sequence: X}
0 \lr \big( W \oplus  \mO_B \big)(2o-2 \tau) \stackrel{i_3} \lr
\big[\big(\textrm{S}^3 W \oplus \textrm{S}^2 W \big) \oplus ( W
\oplus \mO_B)\big]
 (2o- 2 \tau) \lr   \widetilde{A}_6 \lr 0.
\end{equation}
Arguing as in \cite[Lemma 6.14]{CaPi06}, we see that the second
component of the map $i_3$ is the identity, hence the exact
sequence \eqref{sequence: X} splits, giving
\begin{equation*}
\widetilde{ A}_6= (\textrm{S}^3 W \oplus \textrm{S}^2 W ) (2o-2
\tau).
\end{equation*}
By Proposition \ref{prop:vec-ell} this in turn implies
\begin{equation*}
\widetilde{ A}_6 = \bigg( W \oplus W \oplus \bigoplus_{i=1}^3  L_i
\bigg)(2o-\tau),
\end{equation*}
hence $h^0(\widetilde{ A}_6)=9$. We have $1$ parameter for $B$,
  at most $1$ parameter for $\xi$, $1$ parameter for
$\tau$ and $8$ parameters from $\mathbb{P}H^0( \widetilde{ A}_6)$.
Therefore $\mathcal{M}_{\textrm{VIIb}}$ has dimension at most
$11$.
\end{proof}

Summing up, we have the following

\begin{corollary}
The moduli space $\mathcal{M}$ of minimal surfaces of general type
with $p_g = 2$, $q = 1$ and $K^2 = 5$ is unirational and contains
at least $2$ irreducible components. Moreover, the dimension of
each irreducible component is either $12$ or $13$, and there is at
most one component of dimension $13$.
\end{corollary}
\begin{proof}
Notice first that $\mathcal{M}_{\textrm{V}, \textrm{gen}}$ is not
contained in the closure of $\mathcal{M}_{\textrm{V},2}$, since in
the former case $\tau$ is a general point, whereas in the latter
$\tau$ is a $2$-torsion point. So $\mathcal{M}$ contains at least
two irreducible components, namely $\mathcal{M}_{\textrm{V},2}$
and the component containing $\mathcal{M}_{\textrm{V},
\textrm{gen}}$. Moreover there is at most one component of
dimension $13$, namely $\mathcal{M}_{\textrm{I}}$.
\end{proof}

It would be desirable to exactly describe all irreducible
components of $\mathcal{M}$ and to understand how their closures
intersect, but we will not try to develop this point here.

\bigskip
\bigskip

Tommaso Gentile, Paolo A. Oliverio, Francesco Polizzi \\

Dipartimento di Matematica, Universit\`{a} della
Calabria, Cubo 30B, 87036 \\
Arcavacata di Rende (Cosenza), Italy\\

\emph{E-mail address:} \\ \verb|gentile@mat.unical.it| \\
\verb|oliverio@unical.it| \\ \verb|polizzi@mat.unical.it|

\end{document}